\documentclass[]{amsart}
\usepackage{color,latexsym,amssymb,amsmath,amsthm}
\usepackage{comment}
\usepackage{listings}
\usepackage{graphicx}

\definecolor{orange}{rgb}{1,0.5,0}

\newtheorem{thm}{Theorem}

\newtheorem{lemma}[thm]{Lemma}

\theoremstyle{definition}
\newtheorem*{defn}{Definition}

\title{A graph theoretical Poincar\'e-Hopf Theorem}
\author{Oliver Knill}
\date{Jan 4, 2012}
\address{
        Department of Mathematics \\
        Harvard University \\
        Cambridge, MA, 02138
        }

\subjclass{Primary:   05C10, 57M15 }
\keywords{Euler characteristic, graph theory, index, Morse, Poincare-Hopf}

\begin{document}
\maketitle

\begin{abstract}
We introduce the index $i(v) = 1 - \chi(S^-(v))$ for critical points 
of a locally injective function $f$ on the vertex set $V$ of a simple graph $G=(V,E)$. 
Here $S^-(v) = \{ w \in S(v) \; | \; f(w)-f(v)<0 \; \}$ is the subgraph of the 
unit sphere $S(v)$ generated by the vertices $\{ w \in V \; | \;  (v,w) \in E \; \}$. 
We prove that the sum $\sum_{v \in V} i(v)$ of the indices always is equal to the 
Euler characteristic $\chi(G)$ of the graph $G$. This is a discrete Poincar\'e-Hopf theorem
in a discrete Morse setting. 
\end{abstract}

\section{Introduction}

The classical Poincar\'e-Hopf theorem $\chi(M) = \sum_{x} {\rm ind}(x)$ 
for a vector field $F$ on a manifold $M$ plays an important role in differential topology.
Developed first in two dimensions by Poincar\'e \cite{poincare85} (chapter VIII) and extended 
by Hopf in arbitrary dimensions \cite{hopf26}, it is today pivotal in the proof of Gauss-Bonnet theorems
for smooth Riemannian manifolds (i.e. \cite{Guillemin,Spivak1999,Hirsch,Henle1994,docarmo94,BergerGostiaux}. 
The Poincar\'e-Hopf formula has practical value because it provides a fast way to compute the Euler 
characteristic $\chi(M)$ of a compact manifold; it is only necessary to realize a 
concrete vector field, isolate the stationary points and determine their indices. 
This establishes quickly for example that
$\chi(S_{n-1}) = 1-(-1)^n$ because odd dimensional spheres allow vector fields without critical points 
and even dimensional spheres have gradient fields $f$ which have just maxima and minima.
When searching for an analogue theorem for graphs $G=(V,E)$, one is tempted to look for a result which applies
for discrete analogues of vector fields on graphs. But already simple examples like the triangle $G=K_3$ are 
discouraging: a vector field on the edges which circles around the triangle has no critical point even 
so $\chi(G)=3-3+1=1$. 

\begin{figure}
\scalebox{0.30}{\includegraphics{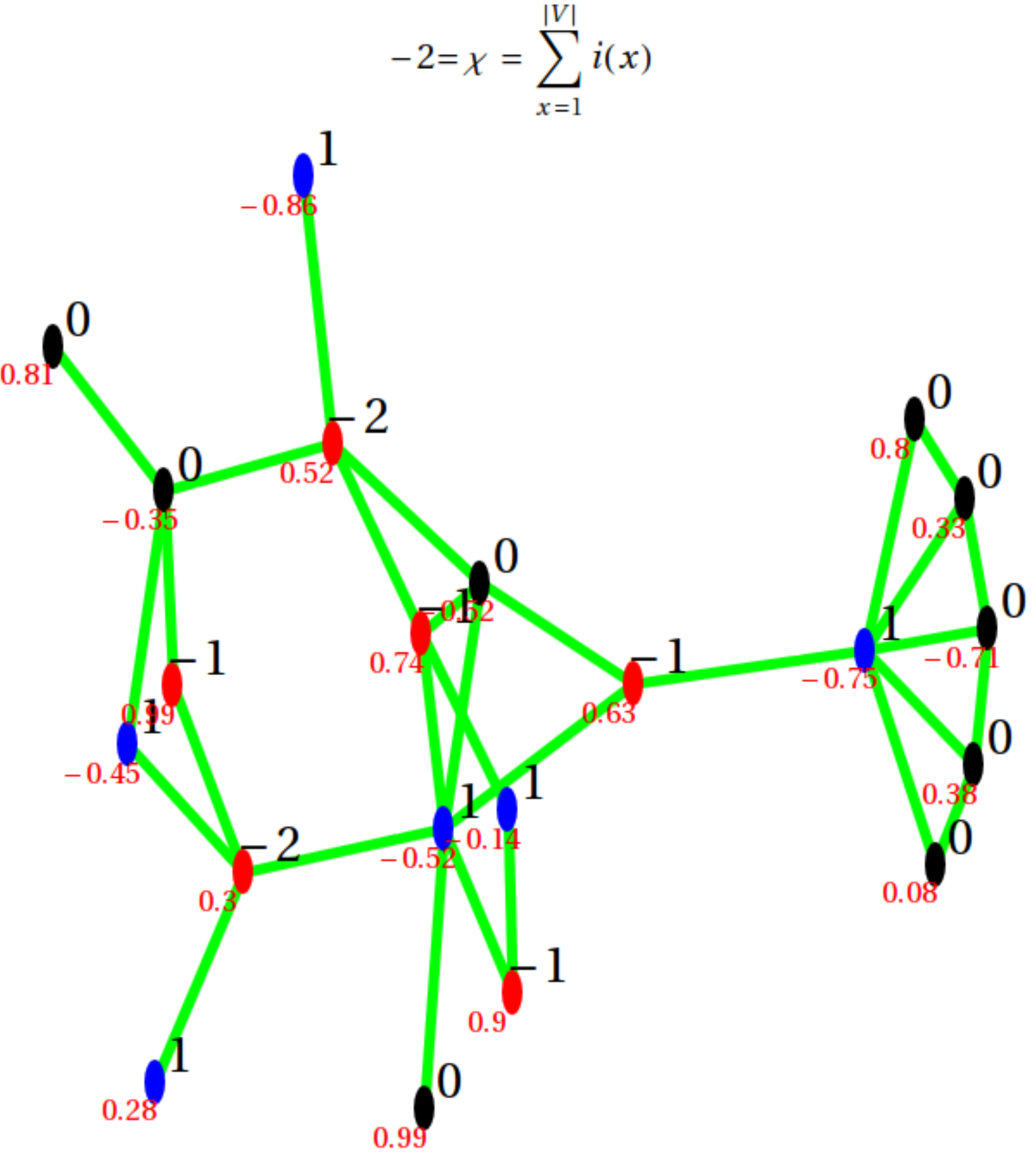}} 
\caption{
The picture shows a graph $G=(V,E)$ with order $|V|=21$ and size $|E|=31$. 
The values of a random Morse function $f$ as well as the 
indices $i_f(v)$ at each vertex $v$ are seen. 
The indices add up to the Euler characteristic, which is in this case $-2$. 
Since there are no tetrahedra $K_4$ as subgraphs, the Euler characteristic is 
$\chi(G) = v-e+f$ where $f$ is the number of triangles, $8$ in this case. 
The nonzero indices in this particular case are $1,1,-1,-2,1,-2,-1,1,1,-1,-1,1$
which add up to $-2$. The vertex on the top is the place where this particular $f$
has attained its maximum. The index is $1$. 
}
\label{illustration}
\end{figure}

Since the vector field encircles a critical point located on the face of the triangle one would be forced to attach 
indices to faces of the graph. For classes of graphs like triangularizations of manifolds, one can indeed proceed like this
and formulate a discrete analogue but it would not be much different from the classical case. 
Gradient fields avoid the contribution of faces. It is Morse theory which allows to avoid vector fields and 
deal with functions alone. We look at a discretization of indices at each of the finitely many 
non-degenerate critical points of $f$. The gradient vector fields are enough for Poincar\'e-Hopf and easier to
discretize because if the curl is zero, one avoids having to deal indices on the faces. 
Morse theory leads to more than Euler characteristic alone. It builds ties
with homology through the Morse inequalities which give bounds for the Betti numbers. Solving a variational problem among
all Morse functions allows to compute Betti numbers and verify the Euler-Poincar\'e formula $\chi(G)=\chi'(G)$
where $\chi(G) = v_0-v_1+v_2+ \cdots$ and $\chi'(G) = b_0-b_1+b_2 + \cdots$. 
The symmetry $f \to -f$ establishes Poincar\'e duality for geometric graphs but the later can not hold for a general graph. 
For a triangle $G=K_3$ for example, where the Betti numbers are $b_0=1,b_1=0,b_2=0$ and the
Euler characteristic is $\chi(G) = b_0-b_1+b_2=v-e+f=3-3+1=1$, Poincar\'e duality does not hold.
Duality holds for geometric graphs including graphs whose local dimension is the same at
all points and for which unit spheres resemble usual spheres. The triangle $K_3$ is a two-dimensional graph but 
every point is a boundary point because the unit sphere is $K_2$ and not a one dimensional circular graph. In higher
dimensions, graphs are geometric if all vertices have the same dimension and every unit sphere is a sphere type graph. The smallest
two dimensional one is the octahedron, where $b_0=b_2=1,b_1=0$ and $\chi(G) = b_0-b_1+b_2=v-e+f=6-12+8=2$.
We will see here that as in the continuum, we can compute $\chi(G)$ also by adding up indices of gradient fields 
of Morse functions. This is most convenient since counting cliques is NP hard and computing cohomology groups 
or rather the dimensions of kernels of huge matrices of total rank $2^{|V|}$ 
- Hodge theory reduces to linear algebra on graphs - is rather tedious. \\

To motivate the definition of the index in the discrete, look at two-dimensional graphs, 
where each unit sphere is a circular graph. At local maxima or minima, $f(y)-f(x)$ has a definite sign on the circle
$S_r(x)$ and the index is $1$. At a saddle point, there are two components $S_r^{\pm}(x)$ of $S_r(x)$, where 
$f(y)-f(x)$ is positive or negative. This goes over to the discrete: in two dimensions, one can define 
$i(x) = 1-s(x)/2$, where $s(x)$ counts the number of sign changes of $f$ on the unit circle. 
Experiments in higher dimensions show that looking at the number of components of $S^{\pm}$ does not work 
any more. It is the Euler characteristic of the 
{\bf exit set} $S^-(v) = \{ y \in S(v) \; | f(y)-f(x)<0 \; \}$ which will matter, 
where we understand the Euler characteristic of the empty set to be $0$. 
We define therefore the index of $f$ at a vertex $v$ as 
$$    i(v) = 1-\chi(S^-(v)) \; . $$  
It will lead to the discrete Poincar\'e-Hopf formula
$$  \sum_{v \in V} i(v) = \chi(G) \; . $$
It speeds up the computation of $\chi(G)$ even more dramatically than 
Gauss-Bonnet 
$$  \sum_{v \in V} K(v) = \chi(G) $$ 
for a local curvature $K(v)$, a relation which also holds for general finite simple graphs. 
By choosing Morse functions in a smart way assigning high values to high degree vertices
and using the method inductively also for the computation of the indices which 
involves Euler characteristic, it allows us to compute $\chi(G)$ in polynomially 
many steps for most graphs, even so counting cliques $K_n$ in $G$ is known to be NP hard. 
As far as we know, it is unknown until now whether Euler characteristic of graphs can be 
computed in polynomial time. While graphs of high dimension and high degree close
to complete graphs could render $\chi(G)$ outside $P$ it is likely that a polynomial $p$
exists such that for most graphs in an Erd\"os-R\'enyi probability space \cite{bollobas}, 
$\chi(G)$ can be computed in polynomial time using the Poincar\'e-Hopf formula.

\section{Morse functions on graphs}

\begin{defn}
A scalar function on a graph $G=(V,E)$ is a function on the vertex set $V$.
The {\bf gradient} of a scalar function at a vertex $x$ is a function on the unit sphere 
$S(x)$ which assigns to a point $y \in S(x)$ with $(x,y) \in E$ the value $f(y)-f(x)$. 
We could write $D_yf(x) = \nabla f(x)(y) = f(y)-f(x)$ to get notation as in the continuum. 
\end{defn}

{\bf Remarks.} \\
1) The analogue in the continuum is the {\bf directional derivative} $D_vf(x) = \nabla f(x) \cdot v$ which is for fixed $x$
a function on the unit sphere of a point and which encodes all information about the gradient. 
Analogously to the continuum, where a differential $1$-form $\alpha$ is a linear function on the tangent space of a
point $x$, the gradient is determined by evaluating it on vectors of length $1$, the unit sphere, 
we think of the gradient of a graph as a scalar function on the unit sphere $S_1(x)$ of a vertex $x$. \\

\begin{defn}
A scalar function $f$ on a graph $(V,E)$ is called
a {\bf Morse function} if $f$ is injective on each unit ball 
$B(p) = \{ p \; \} \cup S(p)$ of the vertex $p$.
\end{defn}

{\bf Remarks.} \\
1) Any function which is injective on the entire vertex set $V$ is of course a Morse function. 
On the complete graph $K_n$, Morse functions agree with the injective functions. \\
2) Morse proved in 1931 \cite{Morse31}, 
that Morse functions are dense in $C^{\infty}(M)$. (Morse deals with analytic functions but they are dense)
The analogue statement is trivial for graphs because the set of all real valued functions on the 
finite vertex set $V$ is a $|V|$-dimensional vector space and injective functions are the complement of 
$\bigcup_{x \neq y} \{f \; | \;  f(x)=f(y) \; \}$, which is a finite union of sets of 
smaller dimensions.  \\
3) A convenient way to chose a Morse function is to enumerate the vertices 
$v_1,v_2, \dots$ and define $f(v_k)= k$.  \\
4) It would also not be enough to assume that $\nabla f$ is not zero at every point because we will 
use later that $f$ restricted to unit spheres are Morse functions as well. 

\begin{defn}
A vertex $x$ of a graph $G=(V,E)$ is called a {\bf critical point} of 
$f$, if the subgraph $S^-(x) = \{ f(p)<f(x) \cap S(x) \; \}$ has Euler characteristic 
different from $1$.
\end{defn}

{\bf Remarks:}  \\
1) Obviously, we could also formulate critical points with the exit set $S^+$ or define to be the condition that 
$\chi(S^+) + \chi(S^-) \neq 2$. We prefer to keep to keep the simpler condition and settle on one side. 
We take $S^-$ to match the continuum case, where $k$ stable directions leads to $S^-(x)=S_k$
and $\chi(S^-(x)) = 1-(-1)^k$ so that $i(x) = (-1)^k$ matches the index of nondegenerate critical points
in the continuum. \\ 
2) This definition works also in the continuum: a point $p$ is a critical point of a
smooth function $f(x,y)$ of two variables if for sufficiently small $r>0$, the 
partition of the spheres $S_r(p)$ into $S_r(p)^-=\{ f(y)-f(x) < 0 \; \}$ 
and $S_r(p)^+=\{ f(y)-f(x) > 0 \; \}$ has the property that both sets are disconnected or
one set is the entire sphere and one is empty. 
This means $i(x)=-1$. The notion $\nabla f(x) =\{ 0 \; \}$ is not a good definition in the discrete
because the local injectivity assumption would prevent that it is ever satisfied for Morse functions. \\
3) For one dimensional graphs, both with or without boundary a point is a critical point if it 
is not a local maximum or local minimum but if just one neighbor is in $S^+$ or just one neighbor is in $S^-$. 
End points are critical points if they are maxima. \\
4) In two dimensions, a point is a critical point if the number $s(x)$ of sign changes of
$\nabla f$ along the circle $S(x)$ is different from $2$. If there are $2$ sign changes, then $S^-(x)$
and $S^+(x)$ are both simply connected one dimensional graphs and have Euler characteristic $1$.  \\
5) For a random function $f$ on a random graph there is a substantial fraction of vertices which are
critical points. Using Monte-Carlo experiments, we measure that the expectation of having a critical point 
when summing over all possible graphs of a given vertex set and all possible Morse functions appears $1/2$. 
This means that for a typical graph and typical Morse function, we expect half the vertices to be 
critical points. 

\section{The index}

\begin{defn}
The {\bf index} of a graph is defined as $i(v)=1-\chi(S^-(v))$, where 
$S^-(v)$ is the subgraph of the unit sphere graph $S(v)$ 
generated by $\{ w \in S(v) \; | \; f(w)-f(v)<0 \; \}$. 
In analogy to the continuum we will refer to $S^-(x)$ also as the {\bf exit set}.
\end{defn}

{\bf Remarks:} \\
1) We could also look at the index $i^+(v)=1-\chi(S_1^+(v))$ or
$j(x) = (i^+(x)+i^-(x))/2 = 1-\chi(S^-(x))/2 - \chi(S^+(x)/2$. Of course, the positive index of $-f$ is the
negative index of $f$. We stick to the simpler index which is computationally half as expensive and which 
matches the continuum, where minima have index $1$ and maxima $(-1)^d$ with dimension $d$. If a graph 
has unit spheres of Euler characteristic $1-(-1)^{d}$ then maxima have index $(-1)^d$. \\
2) The symmetric index $j$ is zero everywhere on circular graphs and for graphs with boundary, 
it is $1/2$ at each end. \\
3) In the classical case, vector fields can have any integer index as the cases
$F(x,y) = {{\rm Re}(x+iy)^n, \pm {\rm Im}(x+i y)^n) }$ illustrate.
Also in the graph case, indices can become arbitrary large. \\
4) For a two-dimensional graph, the index is $i=1-s/2$, where $s$ is the number of 
sign changes of $\nabla f$ on the circle $S_1(x)$. 
This is analogue to the continuum \cite{henle1994} where $i = 1-(h-e)/2$ where $h$ is the number of 
"hyperbolic sectors" and $e$ the number of "elliptic sectors". \\
5) In general, if $f$ has a local minimum $x$, then the index is $1$ because the exit set is empty there.  \\
6) For complete graphs $K_n$, we have for a maximum $0$ for a minimum $1$ as well as for 
any other case. Also for the wheel graph, we can have Morse functions which just have one 
critical point at the center. \\
7) Look at a one dimensional graph with boundary. We have at a minimum $1$ and at a maximum $-1$.
On a minimum at the boundary we have $1$. \\
8) For a tree, the index at a vertex $v$ is $1$ minus the number of neighbors $w$ , where $f(w)$ 
is smaller than $f(v)$. \\
9) The index at a vertex $v$ is bounded below by $1-{\rm deg}(v)$ and bound above by $1+|a|$,
where $a$ is the minimal Euler characteristic of a graph with ${\rm deg}(v)$ vertices. \\
10) The definition works also in the continuum and does not refer to degree or homology:
For a smooth gradient field $F(x) = \nabla f$ in $R^n$ at a critical point $P$, the index is
$1-\chi(S^-(x))$, where $S^-(x)$ is the exit set on a sufficiently small sphere $S_r(x)$. 
This is illustrated in Figure\ref{2dillustration}. 

\begin{figure}
\scalebox{0.14}{\includegraphics{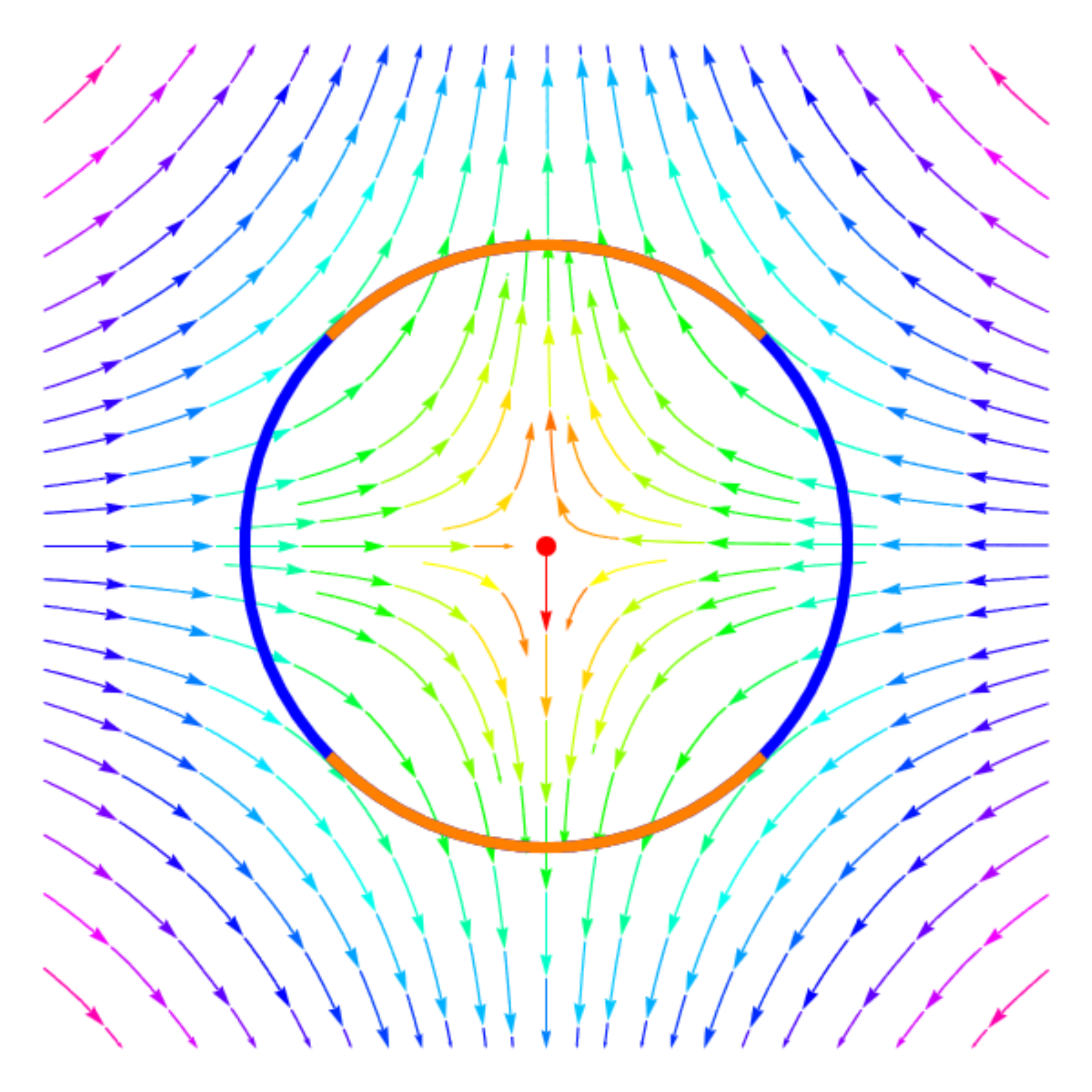}}
\scalebox{0.14}{\includegraphics{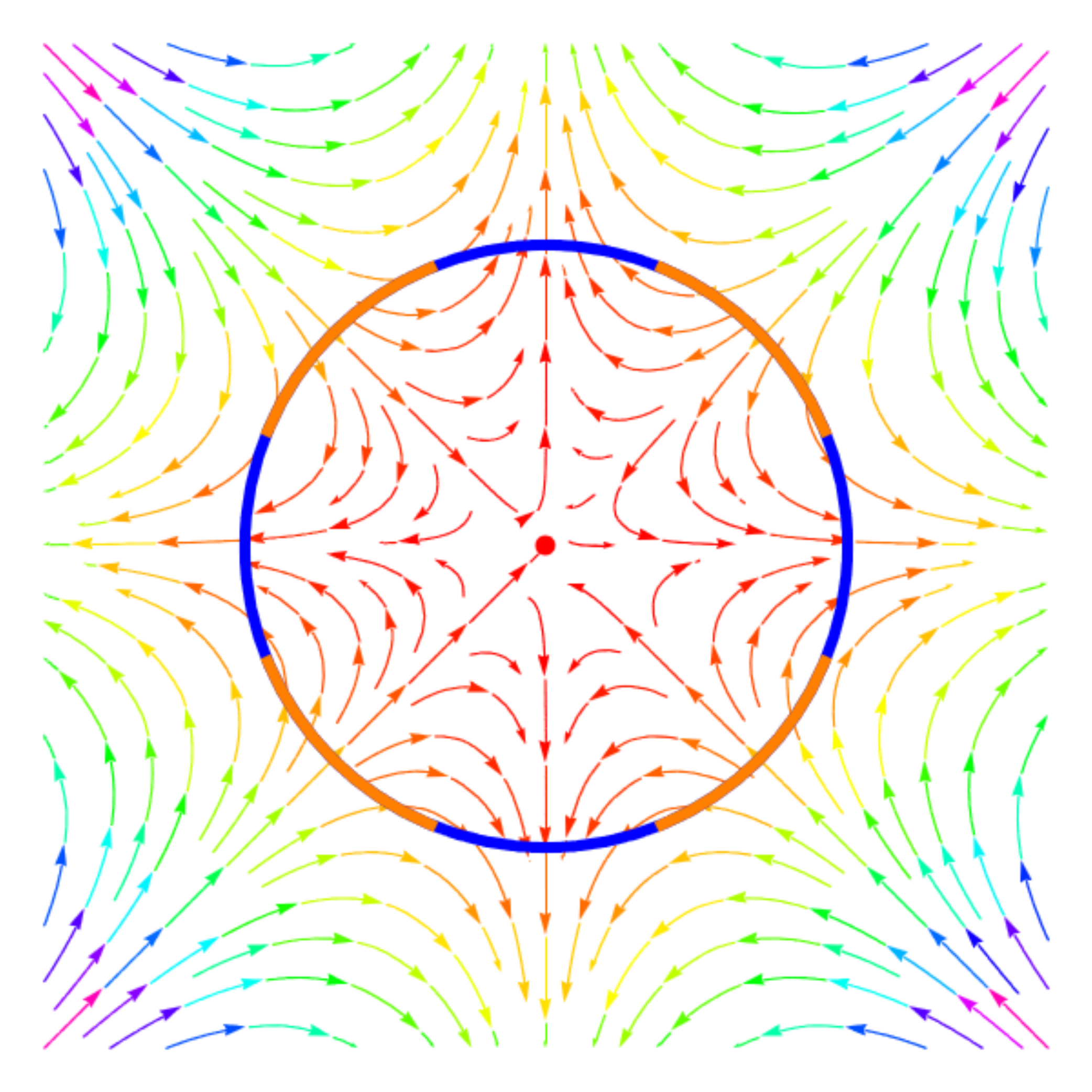}}
\scalebox{0.14}{\includegraphics{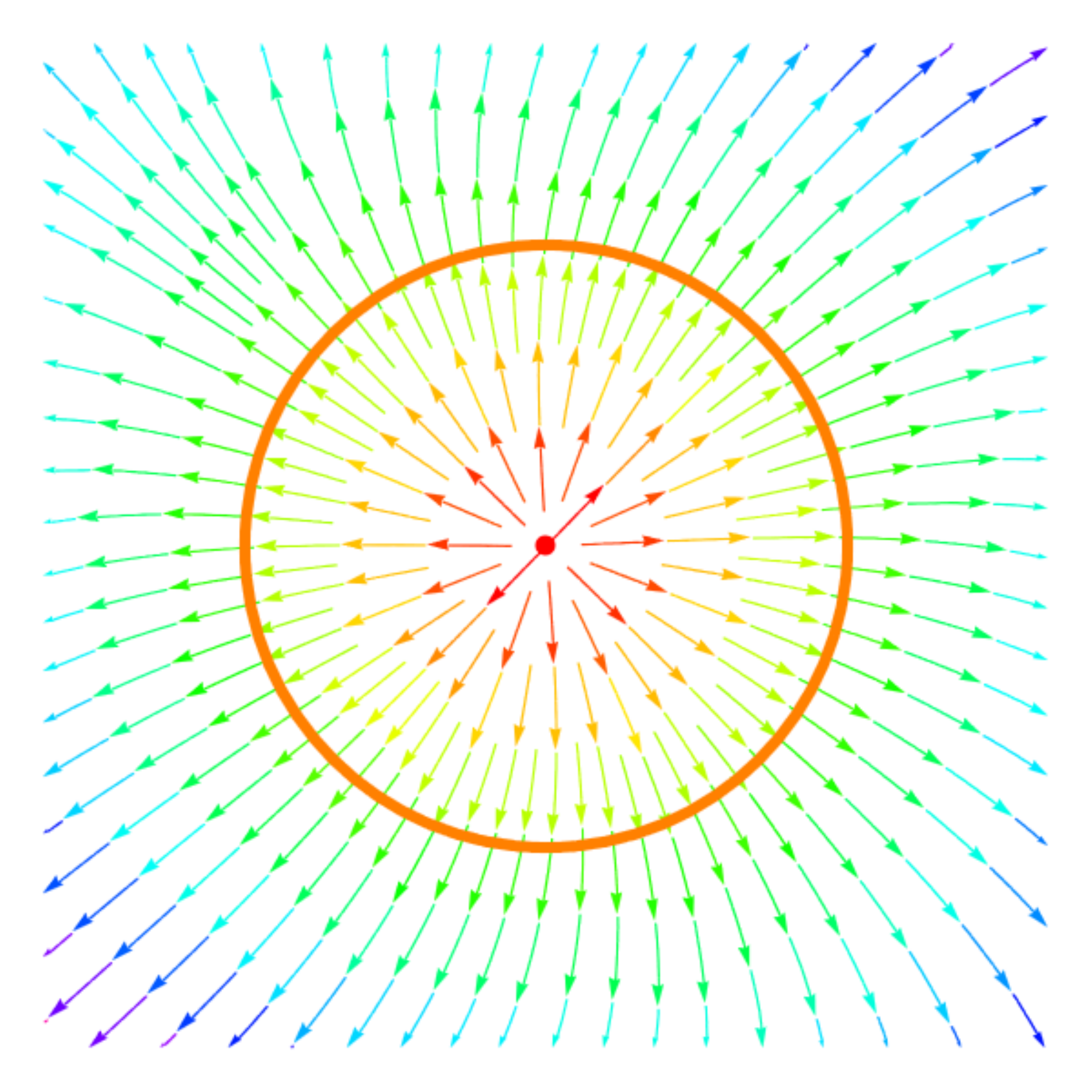}}
\scalebox{0.14}{\includegraphics{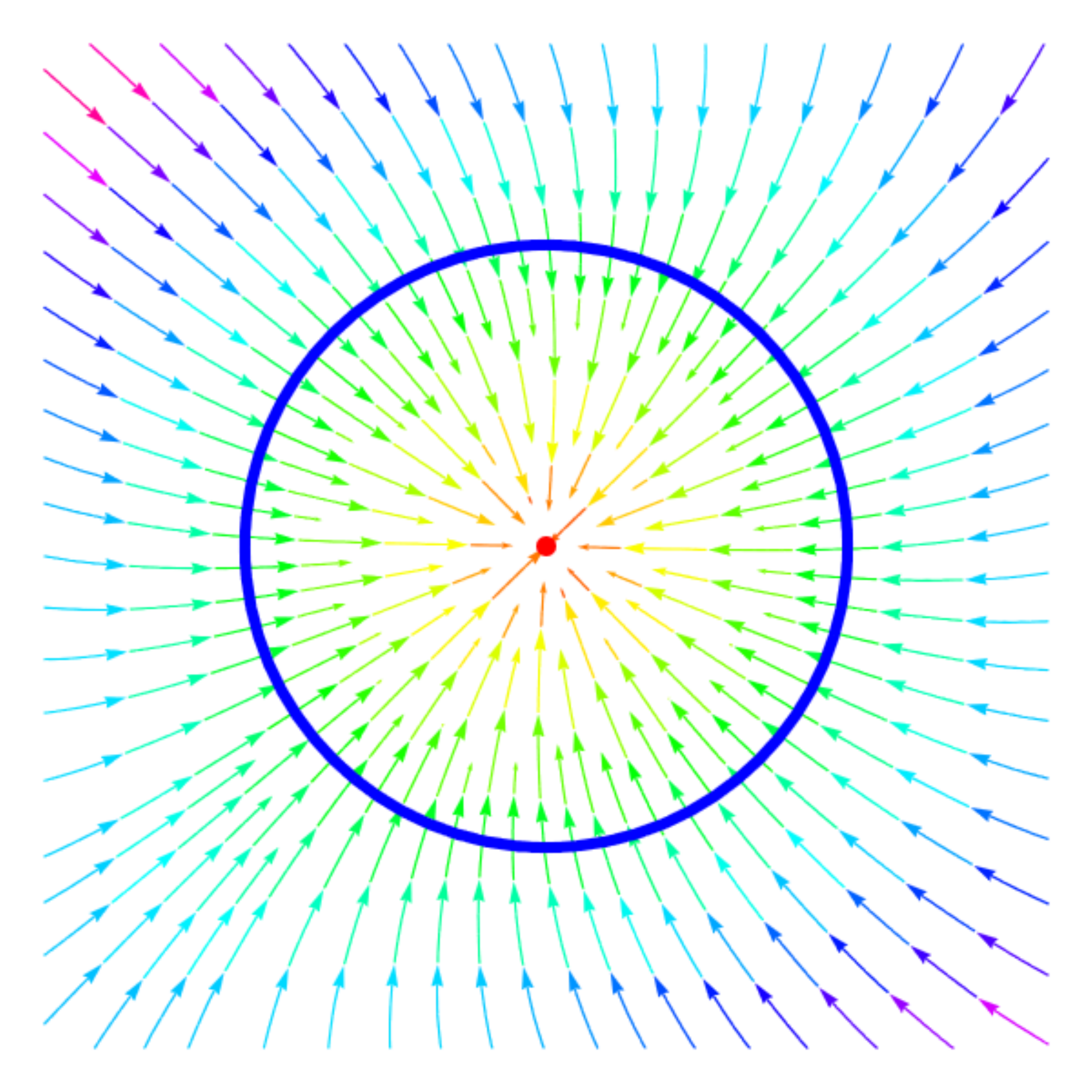}}
\caption{
The exit set $S^-(x) \subset S_r(x)$ is shown for various classical vector fields in the plane
near an equilibrium point. The index of the vector field is $1-\chi(S^-(x))$ if the radius $r$
of the sphere is small enough. The indices in the examples are $-1,-3,1,1$ because $S^(x)$ has
Euler characteristic $2$ (two intervals),$4$ (four intervals),$0$ (empty set),$0$ (circle). 
}
\label{2dillustration}
\end{figure}

\begin{figure}
\scalebox{0.14}{\includegraphics{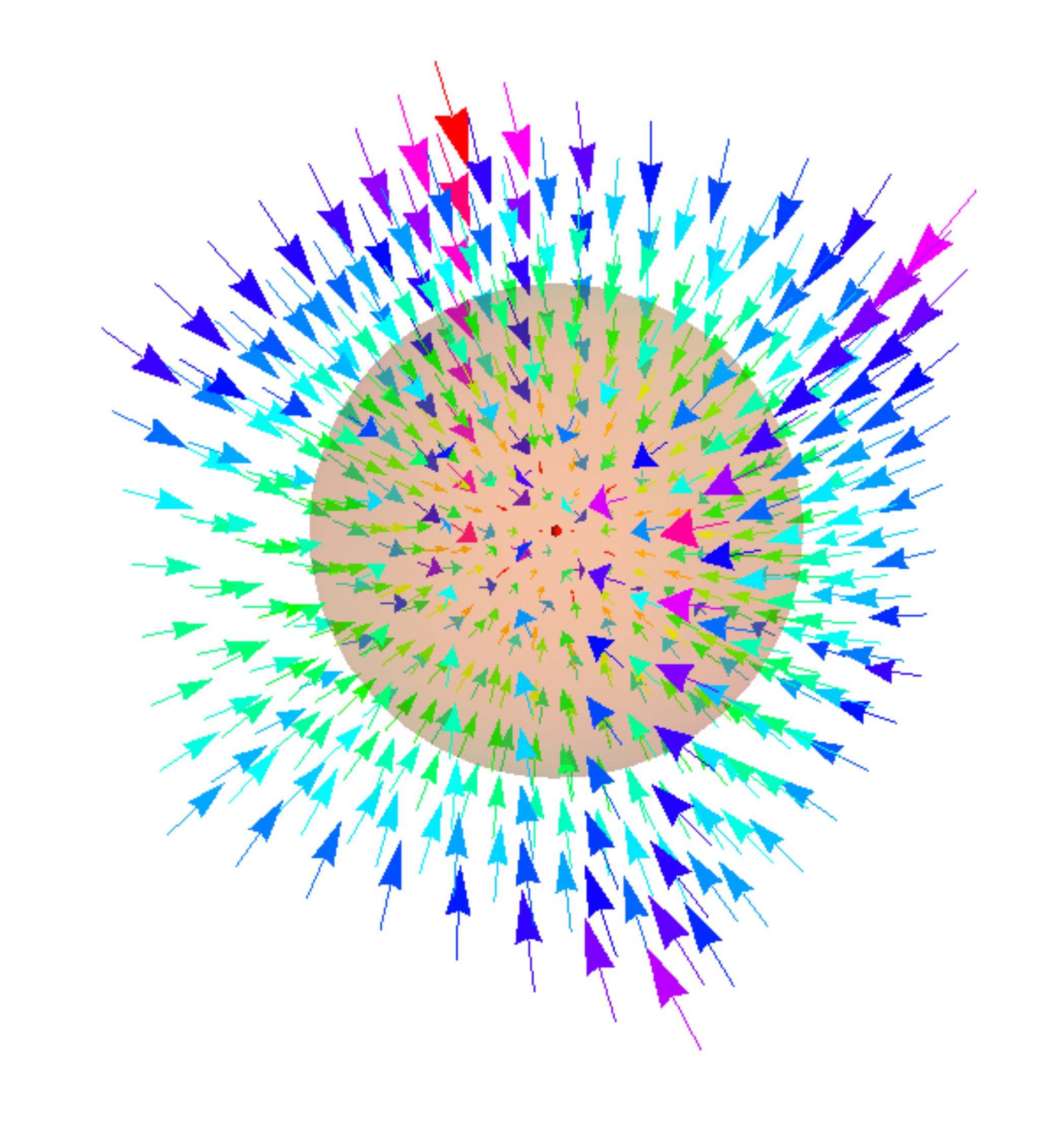}}
\scalebox{0.14}{\includegraphics{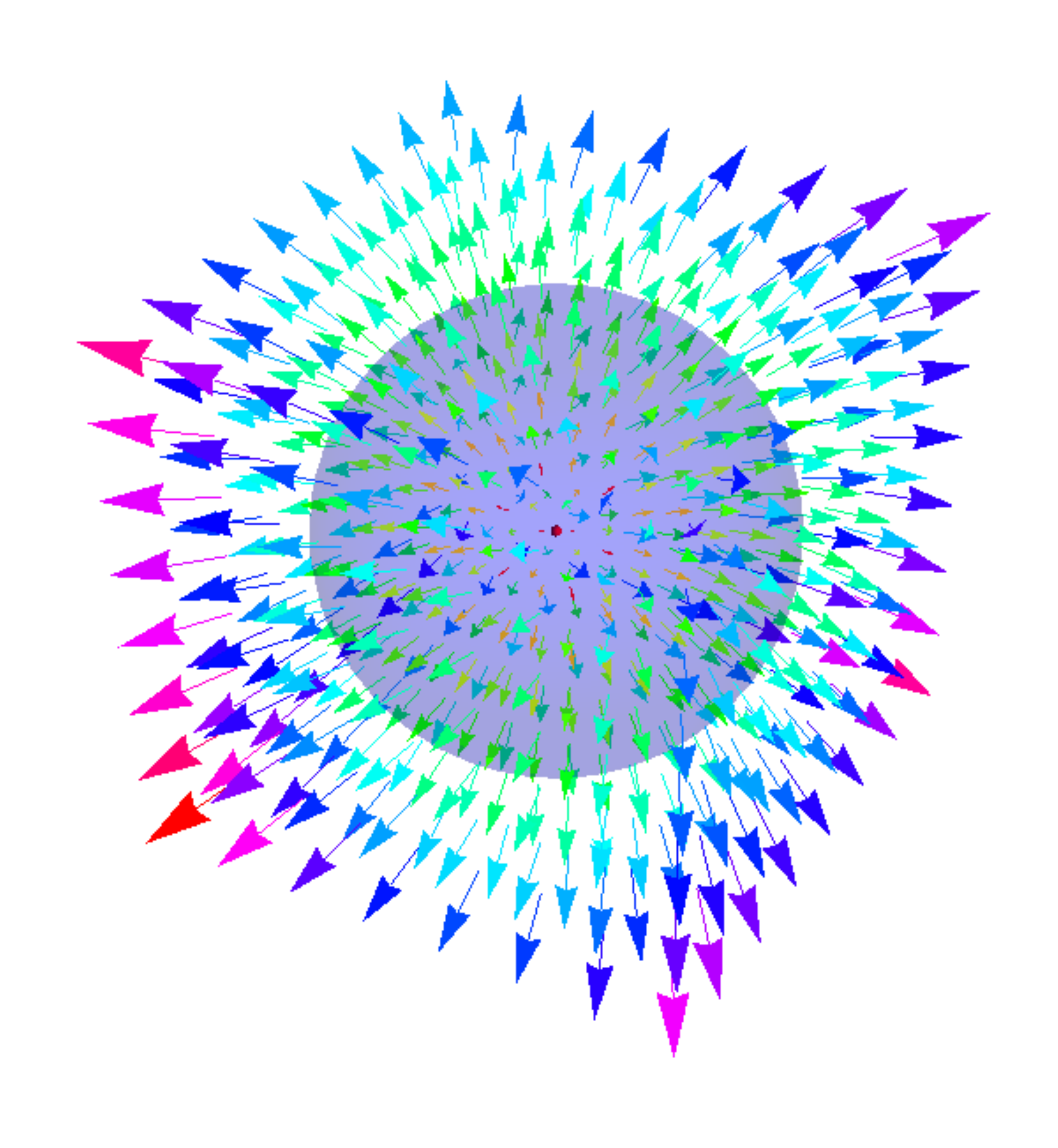}}
\scalebox{0.14}{\includegraphics{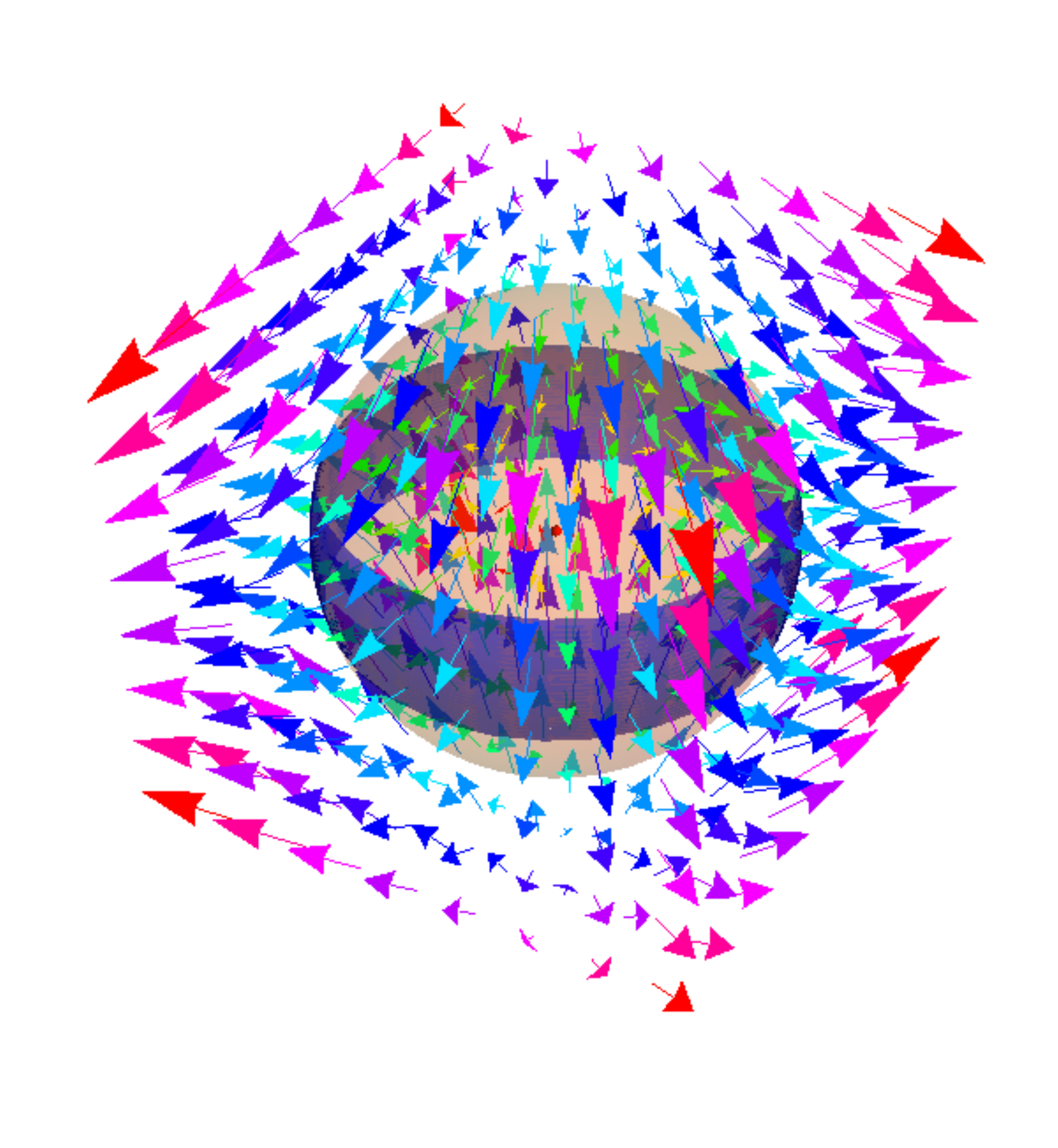}}
\scalebox{0.14}{\includegraphics{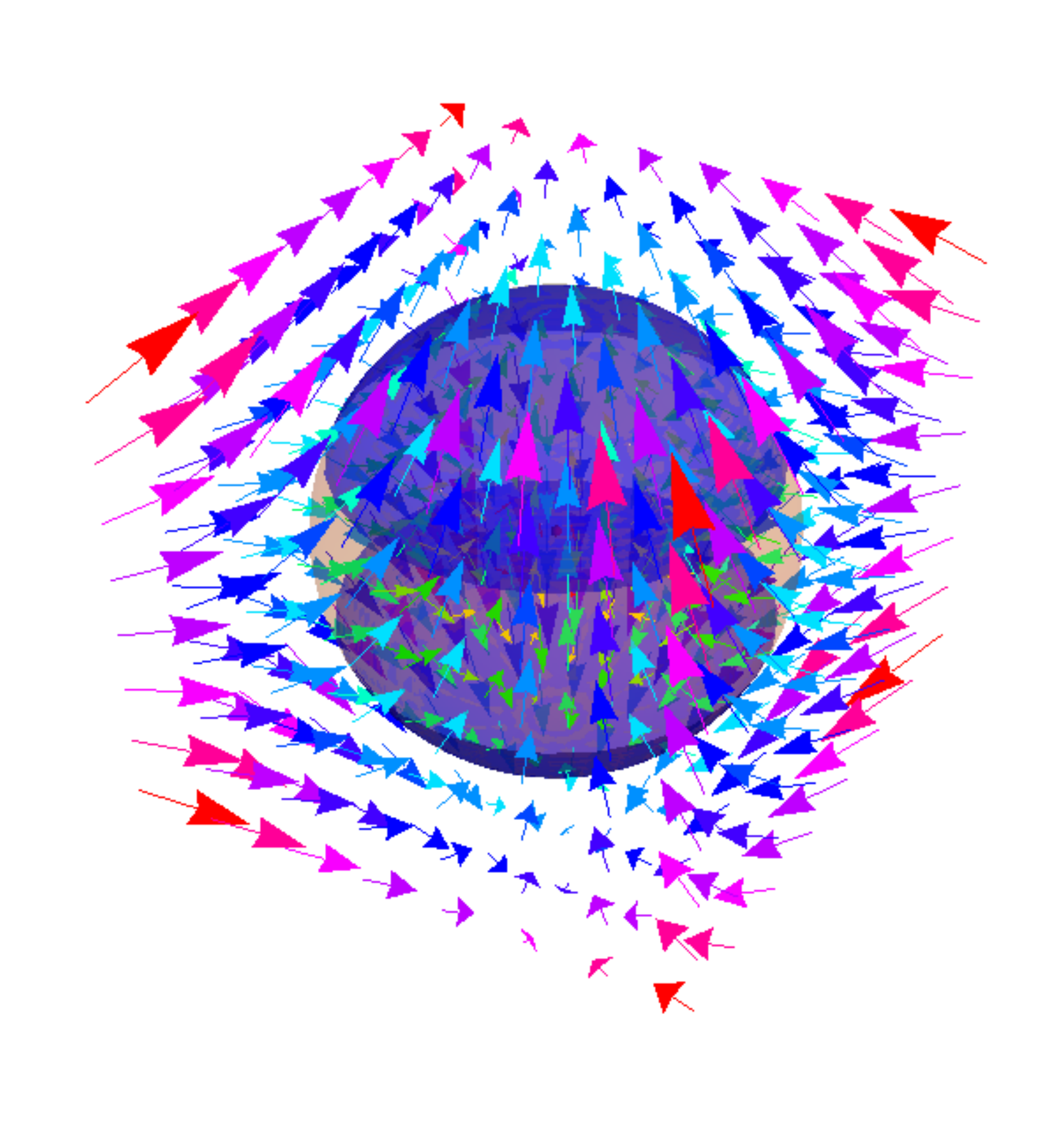}}
\caption{
The exit set $S^-(x) \subset S_r(x)$ for vector fields in space near
an equilibrium point. The index of the vector field is $1-\chi(S^-(x))$ if the radius $r$
of the sphere is small enough. The indices in the examples are $1,-1,1,-1$ because $S^-(x)$ has
Euler characteristic $0$ (empty set),$2$ (sphere),$0$ (cylinder part),$2$ (2 discs).
}
\label{2dillustration}
\end{figure}

\section{Discrete Poincar\'e-Hopf}

\begin{defn}
The Euler characteristic of a graph is defined as 
$\sum_{n=0}^{\infty} (-1)^n v_n$, where $v_n$ is the number of 
$K_{n+1}$ subgraphs of $G$.
\end{defn}

{\bf Remarks}. \\
1. Graphs $K_n$ (or rather the vertex sets alone) are also called cliques in graph theory \cite{handbookgraph,bollbas1,BM}. 
1. In two dimensions, the {\bf Euler characteristic} $\chi(G)$ is $\chi(G) = v-e+f$,
where $v=|V|$ is the number of vertices, $e=|E|$ is the number of edges and 
$f=|F|$ is the number of {\bf faces} in $G$, triangular subgraphs of $G$.  \\
2. A vertex $p$ of a two-dimensional graph is called an {\bf interior point} if $S_1(p)$ is a 
cyclic graph. Otherwise, if $S_1(p)$ is an interval, $p$ is called a {\bf boundary point}. 
The Euler characteristic of a two dimensional graph without boundary points is $b_0-b_1+b_2$
where $b_0=b_2$ agrees the number of components.
If $G$ is a connected two dimensional graph without boundary, 
then $b_0=1$ and $\chi(G)=2-2g$, where $g$ is the genus, the number of holes.
We see here a class of graphs for which the combinatorial Euler characteristic agrees with the
homological definition. \\

Here is the main result of this paper: 

\begin{thm}[Graph theoretical Poincar\'e-Hopf]
Assume $G=(V,E)$ is a simple graph and assume $f$ is a Morse function with 
index function $i=i_f$, then 
$$  \sum_{x \in V} i(x)  = \chi(G) \; . $$
\label{poincarehopf}
\end{thm}

As in the continuum, we use the following lemma:

\begin{lemma}
The index sum is independent of the Morse function.
\label{morselemma}
\end{lemma}
\begin{proof}
Let $M$ denote the set of all Morse functions. 
The index does not change on $M$ unless we pass through a point $f$ in the complement
$$  R = \bigcup_{i, j \in B(i)} \{f \; | \;  f(i)=f(j) \; \} \;  $$
in $M$. We have therefore to understand a single transition, where
we deform the value of the function $f$ at a single fixed vertex $v$
and change the value of the function $f(v)$ such that a single neighboring point $w \in S(v)$,
the value $f(w)-f(v)$ changes sign during the deformation. (The local injectivity can also fail 
if two points in $S(v)$ start to become equal but this does not change $S(v)$. We can therefore
assume $v,w$ are connected.)

Assume the value of $f$ at $v$ 
has been positive initially and gets negative. Now $S^-(v)$ has gained a point $w$ and 
$S^-(w)$ has lost a point.  \\
To show that $\chi(S^-(v)) + \chi(S^-(w))$ stays constant, we check this on
each individual simplex level: 
Let $V_k^+(v)$ denote the number $K_k$ subgraphs of $S(v)$ which connect points
within $S(v)^+$ and let $V_k^-(v)$ connect points which connect vertices within 
$S(v)^-$. We have by definition $i(v) = 1- \sum_k (-1)^k V_k^-(v)$. To prove the lemma,
we have to show that 
$$ V_k^-(v) + V_k^-(w)  $$
stays constant under the deformation. Let $W_k(v)$ denote the number of $K_{k+1}$ subgraphs of 
$S(v)$ which contain $w$. Similarly, let $W_k(w)$ the number of $K_{k+1}$ subgraphs of 
$S(w)$ which do not contain $v$ but are subgraphs of $S^-(w)$ with $v$. 
The sum of $K_k$ graphs of $S^-(v)$ changes by $W_k(w)-W_k(v)$.
When summing this over all vertex pairs $v,w$, we get zero. 
\end{proof}

{\bf Remark. } Lets illustrate the last argument for small $k$: \\
(i) The sum $V_0^+ + V_0^-$ does not change during the deformation 
because one is added on one side and one is lost on the other side. \\
(ii) To the sum of edges $V_1^+ + V_1^-$: let $W_1(v)$ be the number of edges $K_2$
which connect from $S(v)$ to $w$ and let $W_1(w)$ be the number of edges
which connect from $S(w)$ to $v$. The edge sum changes by $W_1(v)-W_1(w)$. 
Take a point $u \in S^-(v) \cap S^-(w)$ which is attached to both $v$ and $w$. 
By definition $f(u)<f(v)$ and $f(u)<f(w)$. 
Initially $(u,w)$ was in $S^-(v)$, after the change it disappeared. 
However the connection $(v,u)$ which was initially not in $S^-(w)$ 
belongs now to $S^-(w)$ afterwards. We see that the sum $W_1(v)$ stays constant. \\
(iii) The sum of triangles $V_2^+ + V_2^-$ changes by $W_2(v)-W_2(w)$
and this is zero after summing over all $v$ because every triangle which was initially in the
graph $S^-(v)$ and loses this property belongs afterwards to $S^-(w)$. 

\begin{lemma}[Transfer equations]
\begin{equation}
  \sum_{v \in V} V_k(v)  = (k+2) v_{k+1} 
  \label{transferequation}
\end{equation}
\end{lemma}

\begin{proof}
See \cite{cherngaussbonnet}. 
For $k=0$ this means $e=v_1=\sum_{v \in V} V_0(v) = \sum_{v \in V} {\rm deg}(v)$ because
every of the two vertices of each edge counts towards the sum of the degree sum. 
For $k=1$ this means that the number $f=v_2$ of triangles is three times the sum of
the number of edges: every triangle has $3$ edges to distribute
to the sum. In general, there are $v_k$ subgraphs $K_{k+1}$ with $k+1$ subgraphs $K_k$
to distribute among the vertex sum. 
\end{proof}

Denote by $W_i(v)$ the number of $K_{k+1}$ simplices in $S(v)$ 
which contain vertices both in $S^-(v)$ and $S^+(v)$. The following lemma is key for
the proof of the discrete Poincar\'e-Hopf theorem \ref{poincarehopf} similarly as the previous lemma \ref{transferequation}
was the core of the proof of Gauss-Bonnet-Chern \cite{cherngaussbonnet}.

\begin{lemma}[Intermediate equations]
\begin{equation}
\sum_{v \in V} W_k(v) = k v_{k+1} 
\label{intermediateequation}
\end{equation}
\end{lemma}

\begin{proof}
(i)  Clearly $\sum_{v \in V} W_0(v) = 0$ because $W_0(v)=0$. \\
(ii) The formula $w_1=\sum_v W_1(v) = v_2$ follows in every of the $v_2$ triangles $(a,v,c)$ in $G$
we have $f(a)<f(v)<f(c)$. This leads to exactly one vertex $(a,c)$ in $E$ which connects $S^-(b)$ with $S^+(b)$. \\
(iii) The equation $\sum_v W_2(v)=2v_3$ follows from the fact that every tetrahedron in $G$
there are two vertices which belong to triangles $(a,v,b)$ where $f(a)-f(v)$ and $f(b)-f(v)$ have
different signs. \\
(iv) In general, if we look at all the $v_{n+1}$ graphs $K_{n+2}$ in $G$, there are $n$ vertices $v$
which are not extremal and which therefore belong to $K_{n+1}$ graphs for which $v$ connects to $S^-(v)$
and $S^+(v)$. 
\end{proof}

Now to the proof of the theorem:

\begin{proof}
Lets call the sum of the indices $\chi'(G)$. Because by lemma~\ref{morselemma},
$f$ and $-f$ have the same sum, we can add up the two equations
\begin{eqnarray*}
    \sum_{v \in V} 1- \chi(S^+) &=& \chi'(G) \\
    \sum_{v \in V} 1- \chi(S^-) &=& \chi'(G) \\
\end{eqnarray*}
and prove
$$ 2 v_0 - \sum_{v \in V} \chi(S^+(v)) + \chi(S^-(v)) = 2 \chi'(G) \;    $$
instead.  Denote by $V_k^+(v)$ the number of $k+1$-simplices $K_{k+1}$ in $S^+(v)$ and with 
$V_k^-(v)$ the number of $k+1$-simplices $K_{k+1}$ in $S^-(v)$ and with $W_i$ the number of 
$K_{k+1}$ simplices in $S(v)$ which contain both positive and negative vertices.  We have 
$$   V_k^-(v) + V_k^+(v) = V_k-(v) W_k(v) \; . $$ 
Let $A(v)$ be the subgraph of $S(v)$ generated by edges connecting $S^-(v)$ with $S^+(v)$. Now
$$ \chi(S(v)) = V_0(v)-V_1(v)+V_2(v) + \cdots  \; ,   $$
$$ \chi(S^+(v)) = V_0^+(v)-V_1^+(v)+V_2^+(v) \cdots  \; ,   $$
$$ \chi(S^-(v)) = V_0^-9v) -V_1^-(v)+V_2^-(v) \cdots  \; ,   $$
$$ \chi(A(v)) = W_1(v)-W_2(v)+ \cdots  \; .   $$
Using the transfer equations~(\ref{transferequation})
and the intermediate equations~(\ref{intermediateequation}), we get
\begin{eqnarray*}
   \chi'(G) &=& 2v_0 + \sum_{k=0}^{\infty} (-1)^k \sum_{v \in V} (V_k^-(v) + V_k^+(v)) \\
            &=& 2v_0 + \sum_{k=0}^{\infty} (-1)^k \sum_{v \in V} (V_k(v)   - W_k(v)  ) \\
            &=& 2v_0 + \sum_{k=0}^{\infty} (-1)^k [ (k+2) v_{k+1} - k v_{k+1}] \\
            &=& 2v_0 + \sum_{k=0}^{\infty} (-1)^k 2 v_{k+1} \\
            &=& 2v_0 - 2v_1 + 2 v_2 - \dots  = 2 \chi(G) \; .
\end{eqnarray*}
\end{proof}

{\bf Examples.} \\
1. Given a tetrahedral graph $K_4$. Every Morse function is equivalent here. 
We can assume $f$ to take values in $[1,2,3,4]$. Now, the index 
is $1$ exactly at the minimum and else zero. This argument works for any 
$K_n$, where we can assume $f(k)=k$ and only the minimum produces an index. \\
2. Given a manifold $M$ with Morse function $f$. A triangularization defines a graph $G=(V,E)$. 
For a sufficiently fine triangularization of $M$, there are the same number of classical 
critical points of $f$ than critical points on the graph and the indices agree. 
The graph theoretical Poincar\'e-Hopf theorem implies the classical Poincar\'e Hopf theorem.  \\

{\bf Remark:} \\
Is there a Poincar\'e-Hopf theorem for vector fields? With a vector field on a simple
graph we mean a directed graph with a function $F:E \to R$ such that $F(e^-)=-F(e)$
where $e^-$ is the reversed edge in the directed graph. Denote by $G={\rm curl}(F)$ the 
completly anti-symmetric function on oriented triangles $t=(e_1,e_2,e_3)$ defined by 
$G(t) = F(e_1)-F(e_2)+F(e_3)$. 
A graph is called simply connected if every closed curve in $G$ can be deformed to a point
using a finite set of deformation steps, where one single deformation step changes the 
path on a triangle $K_3$ subgraph of $G$ only. As in the continuum, if
a graph is simply connected and a vector field $F$ has zero curl, then $F$ is a gradient
field: there is a function $f$ on vertices such that $F((x,y)) = f(y)-f(x)$. Such a vector
field is called conservative. It is called nondegenerate, if the potential $f$ is a
Morse function. Obviously, for nondegenerate conservative vector fields, the theorem
applies too. As mentioned in the introduction, if the curl is not zero on some triangle,
then the theorem would need a modification. One could add $1/3$ to the index of each triangle with
nonvanishing curl and get a Poincar\'e-Hopf result for more general vector fields, at least in
two dimensions. 

\section{Graphs of sphere type}

Reeb's theorem states that if a smooth Morse function on a compact $d$ dimensional manifold $M$ 
has only two critical points and both are non-degenerate, then $M$ is homeomorphic to a sphere
\cite{Mil63}. Here is a weak discrete analogue:

\begin{defn}
Given a simple graph $G=(V,E)$. Denote by $m(G)$ the minimal number of critical points 
among all Morse functions on $G$. 
\end{defn}

{\bf Examples.} \\
1. The invariant $m(G)$ is a positive integer for a graph $G=(V,E)$, if the edge set $E$ is nonempty. 
It is $n$ for graphs $P_n$ of order $n$ and size $0$ which do not have any edges. \\
2. For a one dimensional cyclic graph, it is $2$. 
For a triangularization of a $n$-dimensional sphere it is $2$.  \\
3. For a discrete torus, the minimum can be $3$: the min and max produce index $1$
then we have one with index $-2$. 

{\bf Remarks.} \\
1. Here is a discrete approach to Reeb's theorem:
If $m(G)=2$, and $\chi(S(v))=1-(-1)^{d-1}$ then $G$ is graph with Euler characteristic $1-(-1)^d$. 
Proof: one of the points must be a maximum and have index $(-1)^{d-1}$, the other must be a minimum of index $1$. 
The discrete Morse theorem assures that the Euler characteristic is $1+(-1)^{d-1}=1-(-1)^d$.
2) The one dimensional version of the discrete Reeb theorem is trivial because the condition of
having exactly two critical points implies that the graph has only one connectivity component.
Since a general, a one dimensional graph is a union of one dimensional circles, it must be a simple circle. 
If a one-dimensional graph $G$ without boundary has a Morse function $f$ with exactly two
critical points, then $G$ is a cyclic graph of Euler characteristic $0$.  \\

Of course, we would like to get closer to the standard continuum Reeb theorem.
Because we do not have a notion of a "standard sphere" in graph theory (pyramid constructions
show that the unit sphere of a vertex can be an arbitrary graph) we make the following 
inductive definition:

\begin{defn}
A graph $G=(V,E)$ is of {\bf sphere type}, if it is either the empty graph or if $m(G)=2$ 
and every unit sphere $S(v)$ of $G$ is of sphere type. 
\end{defn} 

{\bf Examples.} \\
1. By definition, a graph $G=(V,E)$ with empty edge set $E$ is of sphere type, if it consists of two isolated points.  \\
2. A graph for which every point has dimension $1$ is of sphere type if it is equal to $C_n$ for $n>3$.  \\
3. A graph for which every point has dimension $2$ is of sphere type if it is a geometric graph for which 
every unit sphere is a circular graph. Examples are the octahedron and the icosahedron. 
Every triangularization of the unit sphere works. Every such geometric two dimensional graph defines a compact
two dimensional surface. If the Euler characteristic of this surface is $2$, then the graph is of sphere type. \\
4. The three dimensional 600 cell is of sphere type because every unit sphere is an icosahedron which is
of sphere type and because we can place it into $R^4$ with a function $r: V \to R^4$ such that $f(v) = (0,0,0,1) \cdot r(v)$
is a Morse function which has just two critical points. \\
5. Every graph which is of sphere type is connected if it has positive size. \\
6. Every graph $G$ of sphere type of positive size we know of for which the unit sphere is connected
is simply connected: every closed curve in $G$ can be deformed to a point where a single deformation 
step changes a path only on a triangle $K_3$. 

\section{Illustrations}

We look first at some one dimensional graphs.   \\

1. For a circular graph $C_n$ and a Morse function $f$, we have $i(x)=-1$ if $S^-(x)$ consists of 
two points, which means that $f$ is a local maximum. We have $i(x)=1$ at points where
$S^-(x)$ consists of zero points, which is a local minima. If $S^-(x)$ consists of one
point, then we have a regular point and the index is zero. Poincar\'e-Hopf on a circular
graph tells that the number of local maxima is equal to the number local minima. \\

\begin{figure}
\scalebox{0.22}{\includegraphics{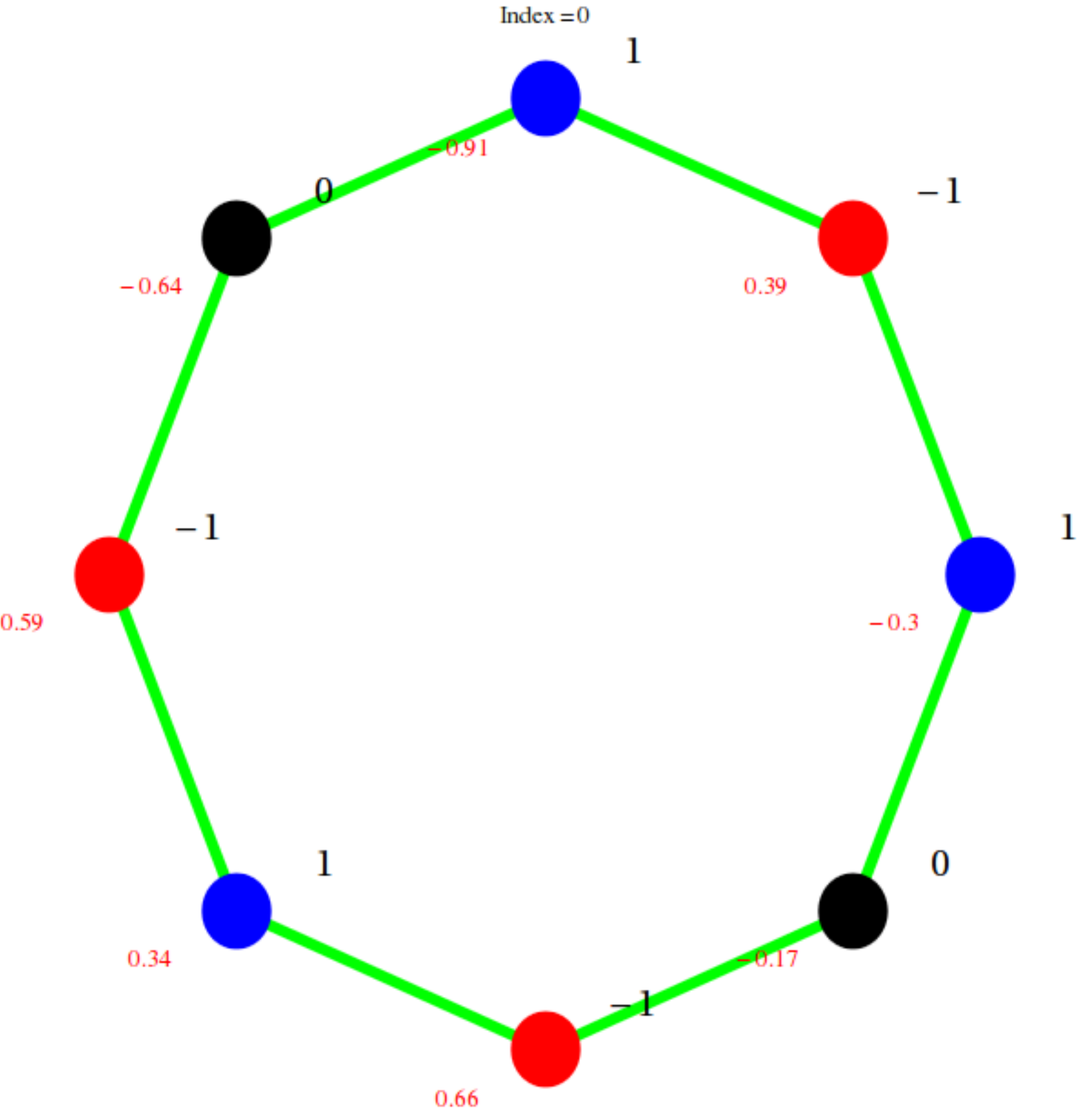}}
\scalebox{0.22}{\includegraphics{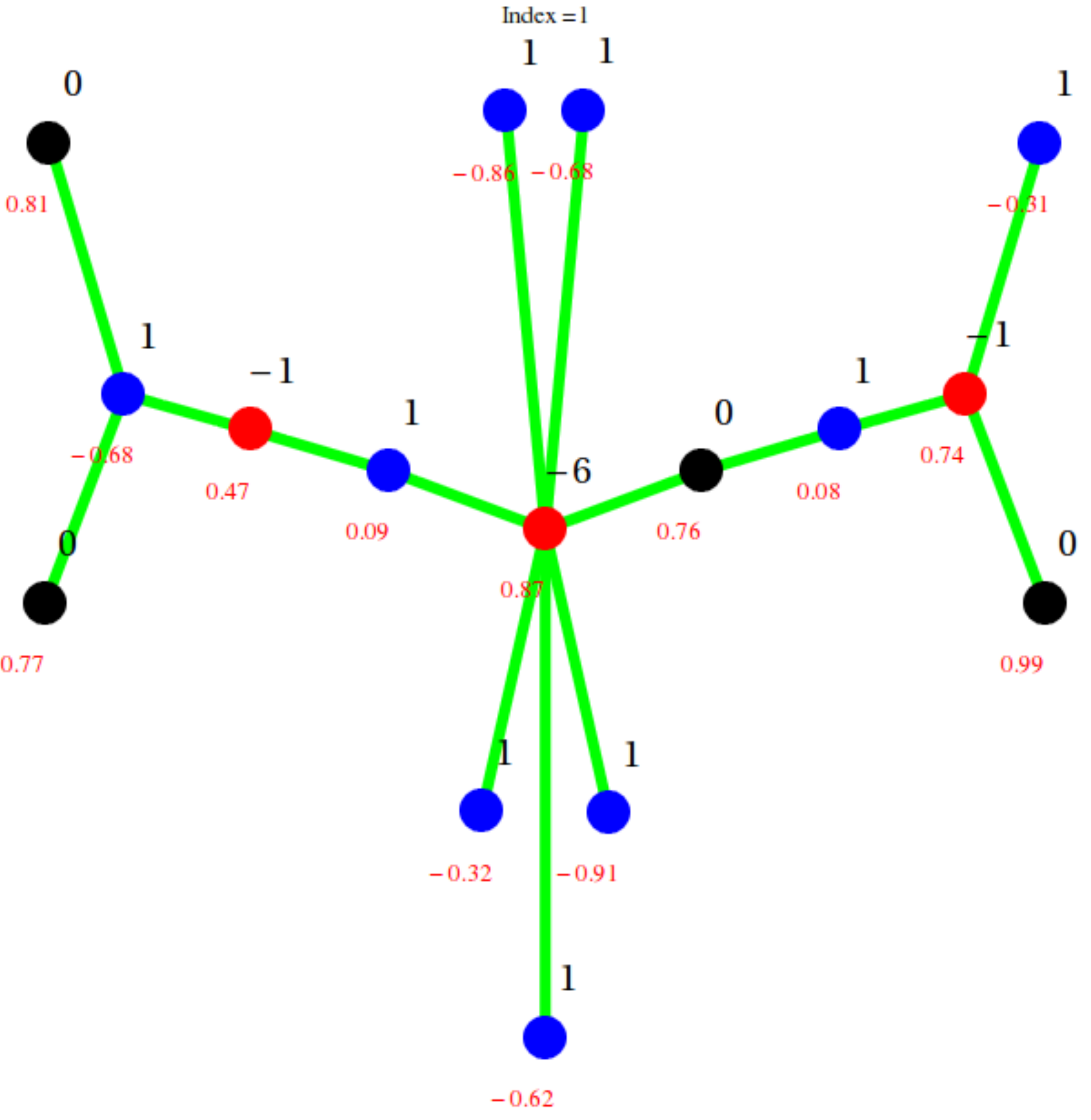}}
\caption{
The left picture shows a cyclic graph. The index at a maximum is $-1$, the index at a minimum is $1$. 
The right picture shows a tree. The index can not become larger than 
$1$ for one dimensional graphs because the Euler characteristic can not become negative for 
zero dimensional graphs. 
}
\label{onedim}
\end{figure}

2. For an interval, a one dimensional graph with two boundary points, the index is $1$
at local maxima. For $f=[1,3,2]$ the index is $i=[0,1,0]$. For $f=[2,1,3]$ the index is $i=[1,-1,1]$.  \\

3. For a tree, $i(x)$ is $1$ minus the number of smaller neighbors. At a minimum, we have
index $1$ at a maximum it is $1$ minus the number of smaller neighbors. \\

4. For an octahedron or icosahedron, we can find Morse functions which have just two critical points. 
Both maxima and minima now have circular graphs as unit spheres. 

\begin{figure}
\scalebox{0.21}{\includegraphics{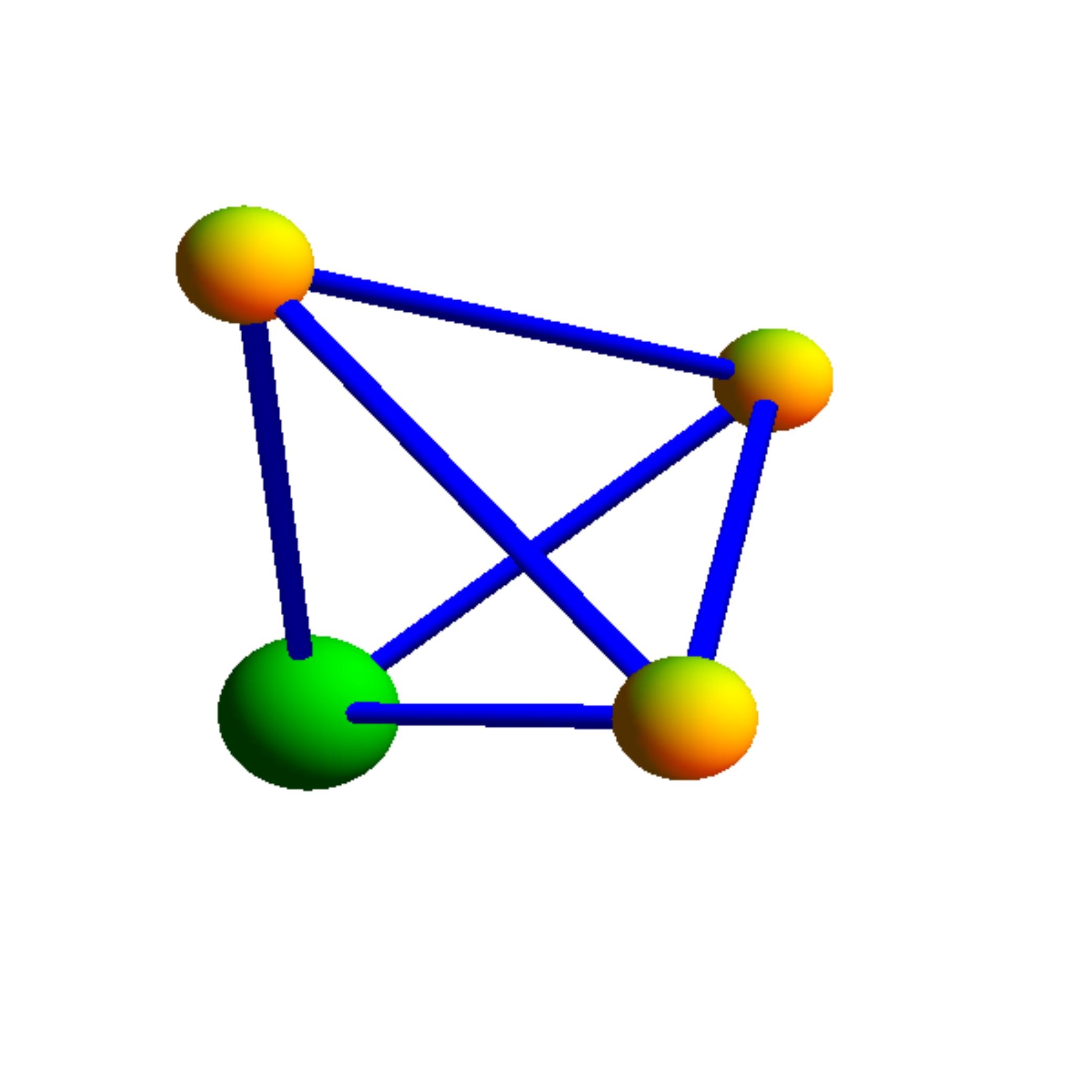}}
\scalebox{0.21}{\includegraphics{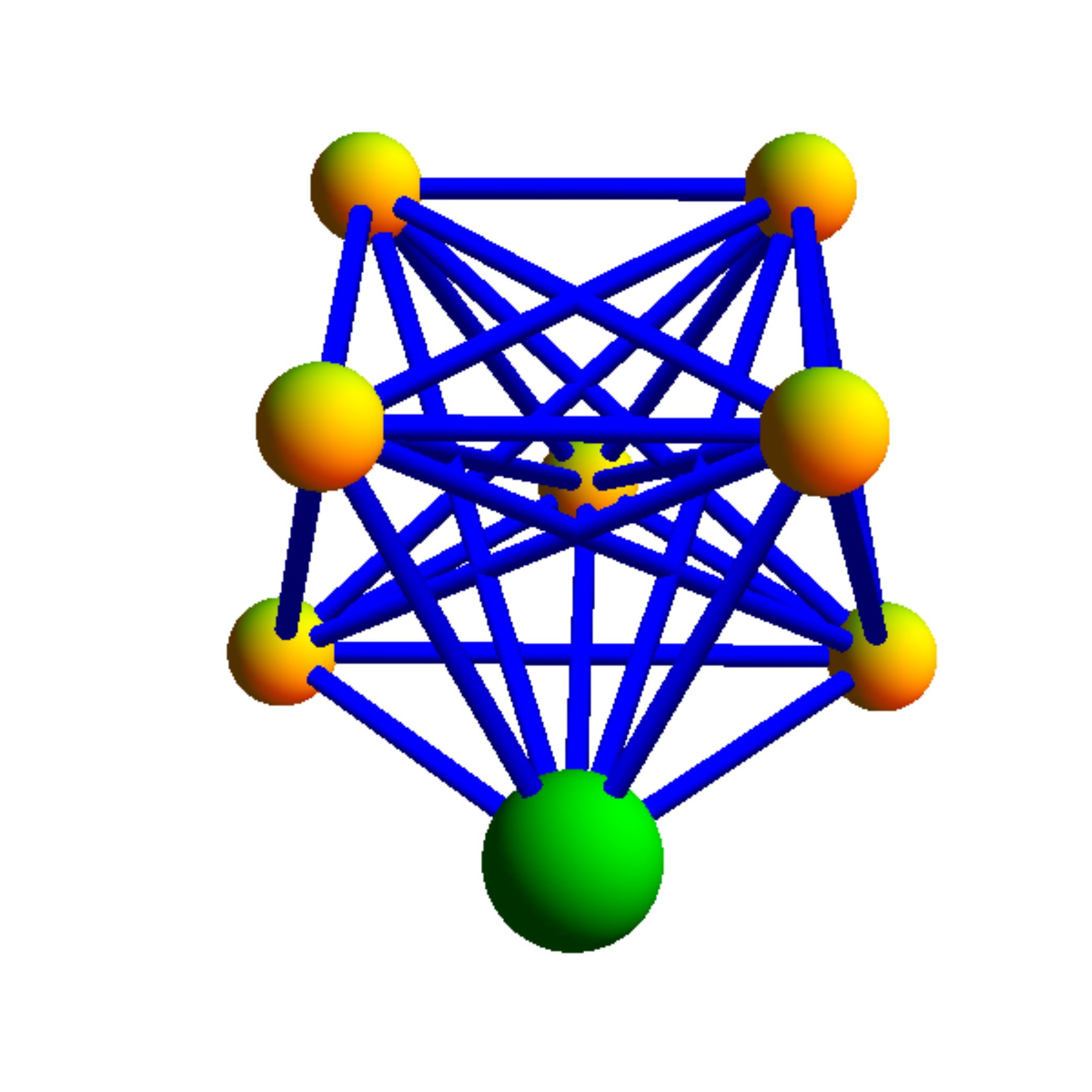}}
\caption{
For the complete graph and any injective function $f$, there is only one
critical point, the maximum, for which $S^-(v)$ is empty so that $i(v)=1$. 
The left picture shows $K_4$, the tetrahedron. The right picture shows 
$K_8$. 
}
\label{complete}
\end{figure}

5. For a three dimensional cross polytope, each unit sphere is an octahedra. At a maximum, 
the unit sphere has Euler characteristic $2$ and the index is $-1$. 
At a minimum, the index is $1$. The sum of the indices is $0$. We know in general \cite{gaussbonnetchern}
that for geometric three dimensional graphs, the Euler characteristic is $0$. 

\begin{figure}
\scalebox{0.24}{\includegraphics{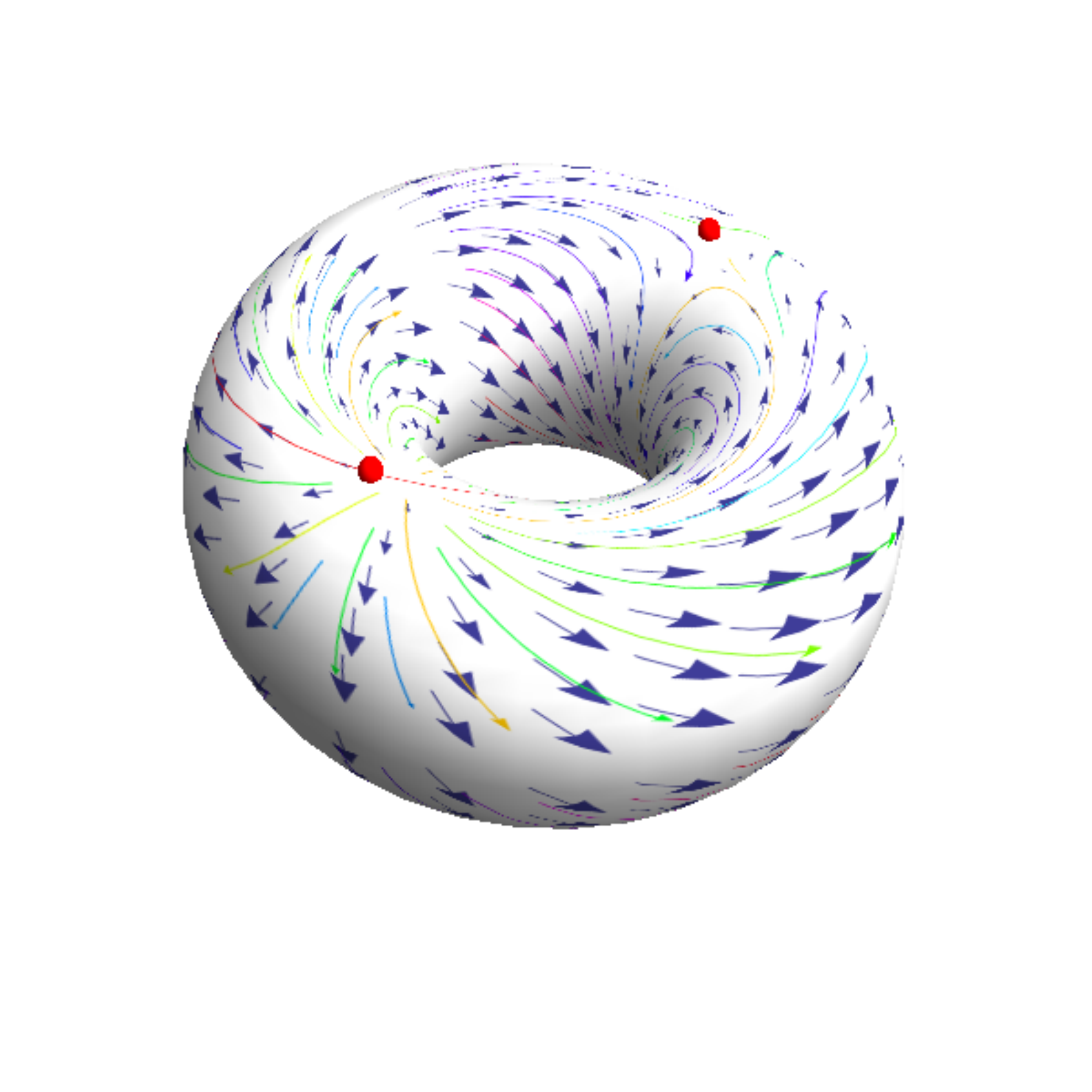}}
\caption{
The figure shows a discrete torus $G=(V,E)$ embedded in space by a map $r: V \to R^3$. 
If all vertices have different height, we can use the 
Morse function $f(v) = (0,0,1) \cdot r(v)$.
There are only 4 critical points. The maxima and minima have index $1$ because the unit sphere
is the circular graph $S(v) = C_6$ and $1-\chi(S^-({\rm max}))=1-\chi(\emptyset)=1$
and $1-\chi(S^-({\rm min})) = 1-\chi(C_6) = 1$. The saddle points have index $-1$ because 
$i\chi(S^-(v))=0$ in those cases. 
}
\label{torus}
\end{figure}

6. For any $d$- dimensional polyhedron $V$ with triangular faces which can be 
realized as a triangularization of a $d$-dimensional sphere and for which the 
unit spheres have the same properties, and so on, we have $\chi(G) = 1 +  (-1)^d$. 
Proof: embed the polyhedron into $R^{d+1}$
in such a way that the $d+1$'th coordinate is injective. 
Then look at $f(x) = x_{d+1}$. This is an injective
function on $V$ and so Morse function on $G$. By induction, the unit sphere 
of a vertex is a $d-1$ dimensional convex polyhedron which is a 
triangularization of $S^{d-1}$. It has Euler characteristic $1-(-1)^{d-1}$. 
The maxima and minima therefore have indices $0$ for odd $d$ and $1$ for even $d$. 

\begin{figure}
\scalebox{0.24}{\includegraphics{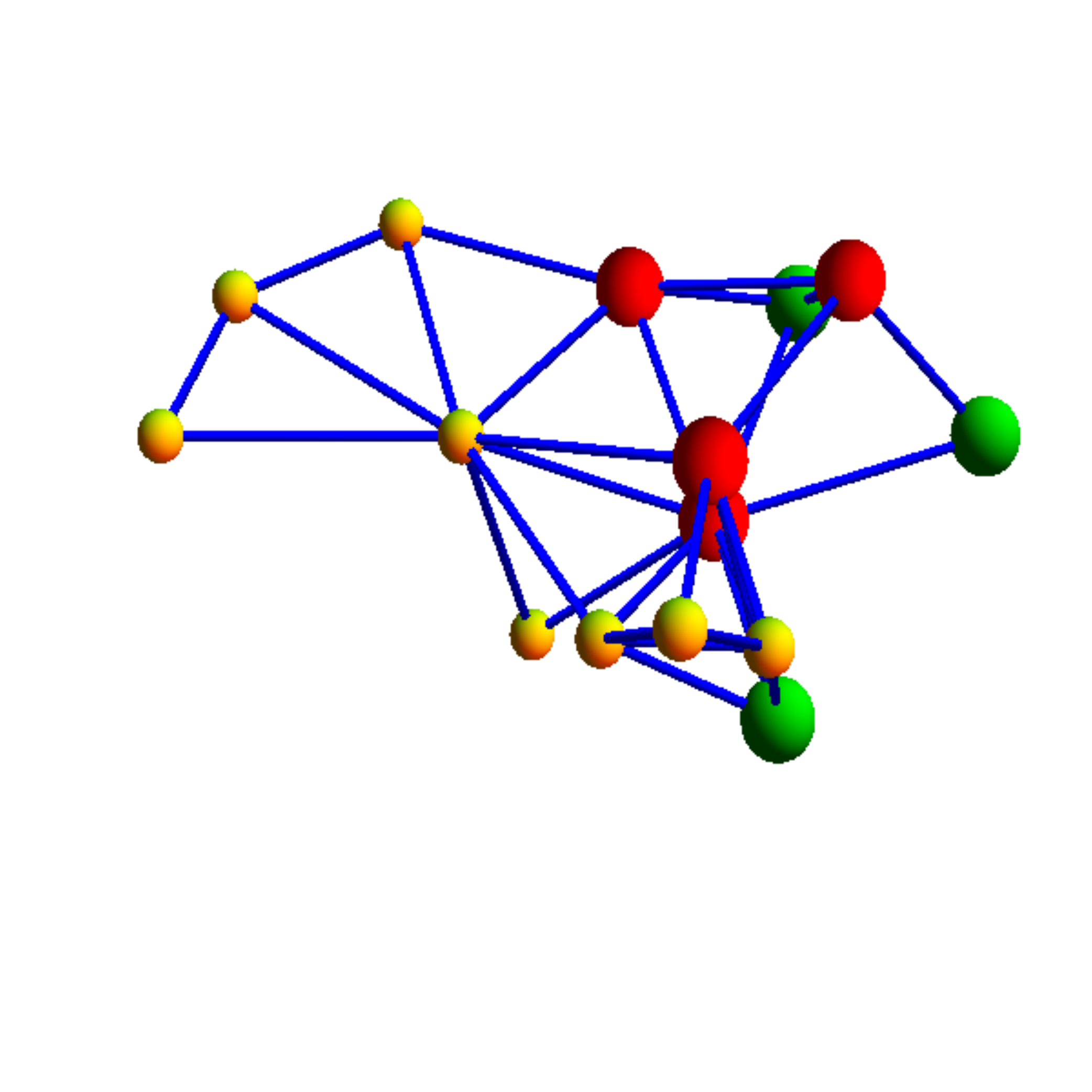}}
\scalebox{0.24}{\includegraphics{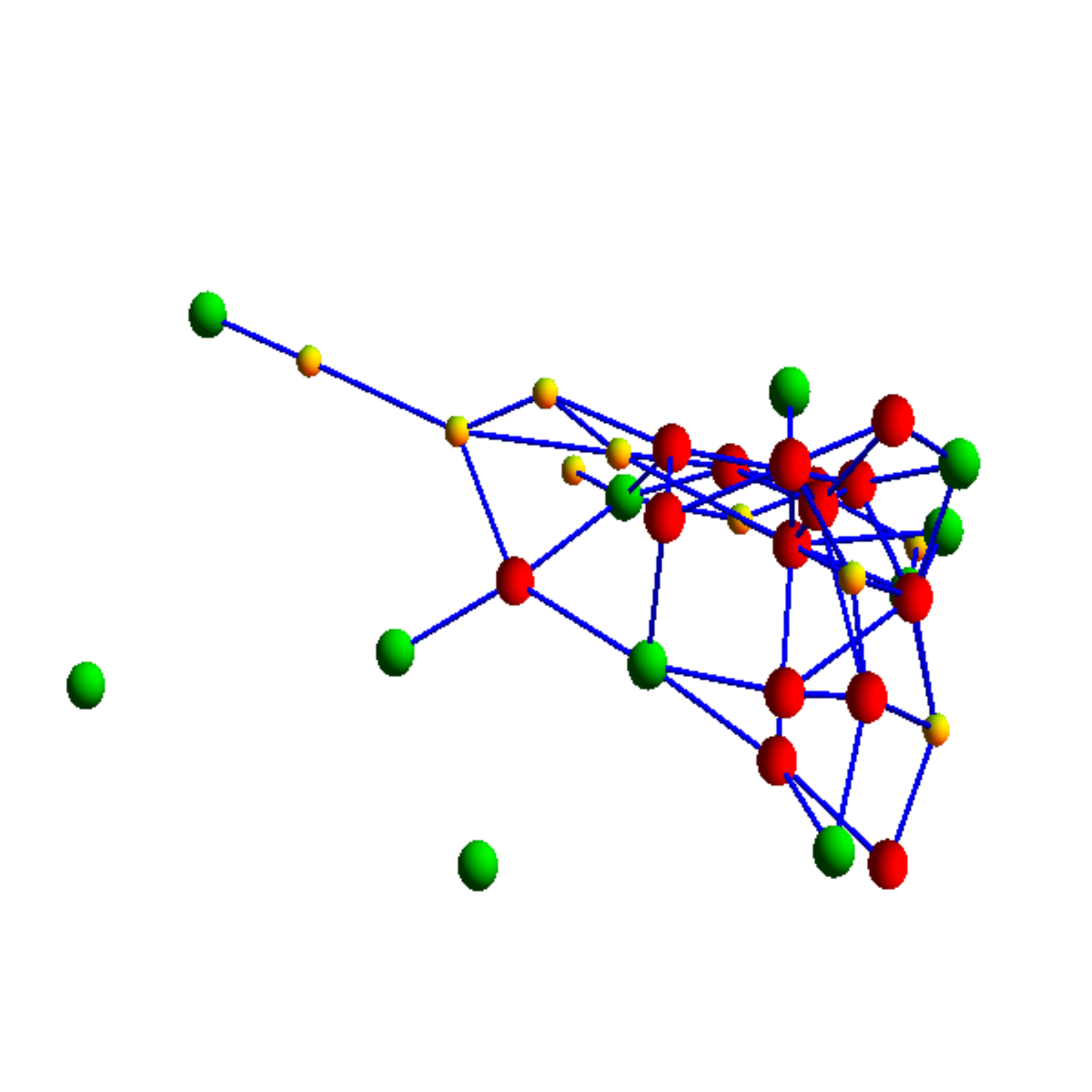}}
\caption{
A random graph with 15 vertices and one with 35 vertices. Also chosen is
a random Morse function $f$. The vertices are larger at critical points, 
points where the indices are nonzero. 
}
\label{complete}
\end{figure}

\begin{figure}
\scalebox{0.21}{\includegraphics{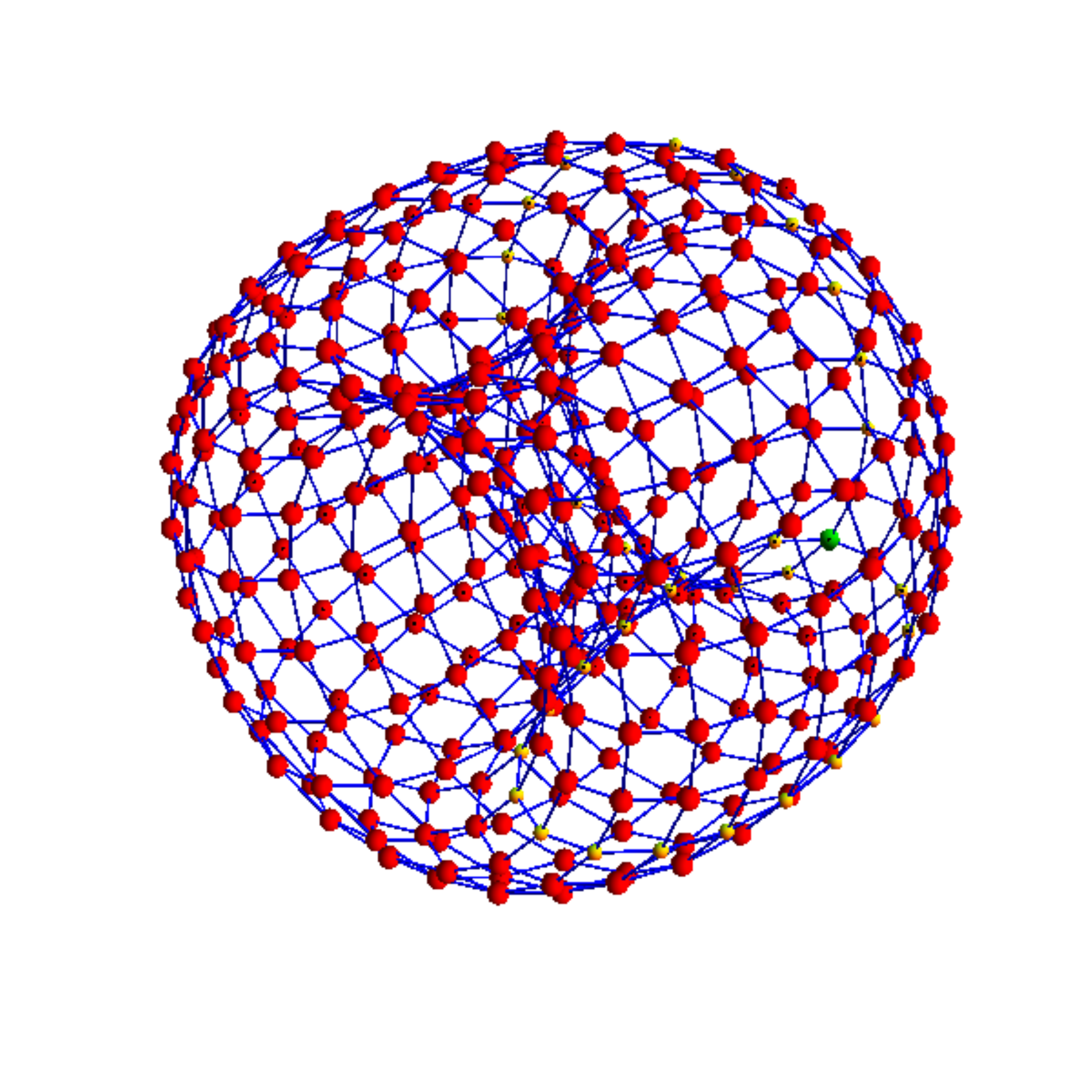}}
\scalebox{0.21}{\includegraphics{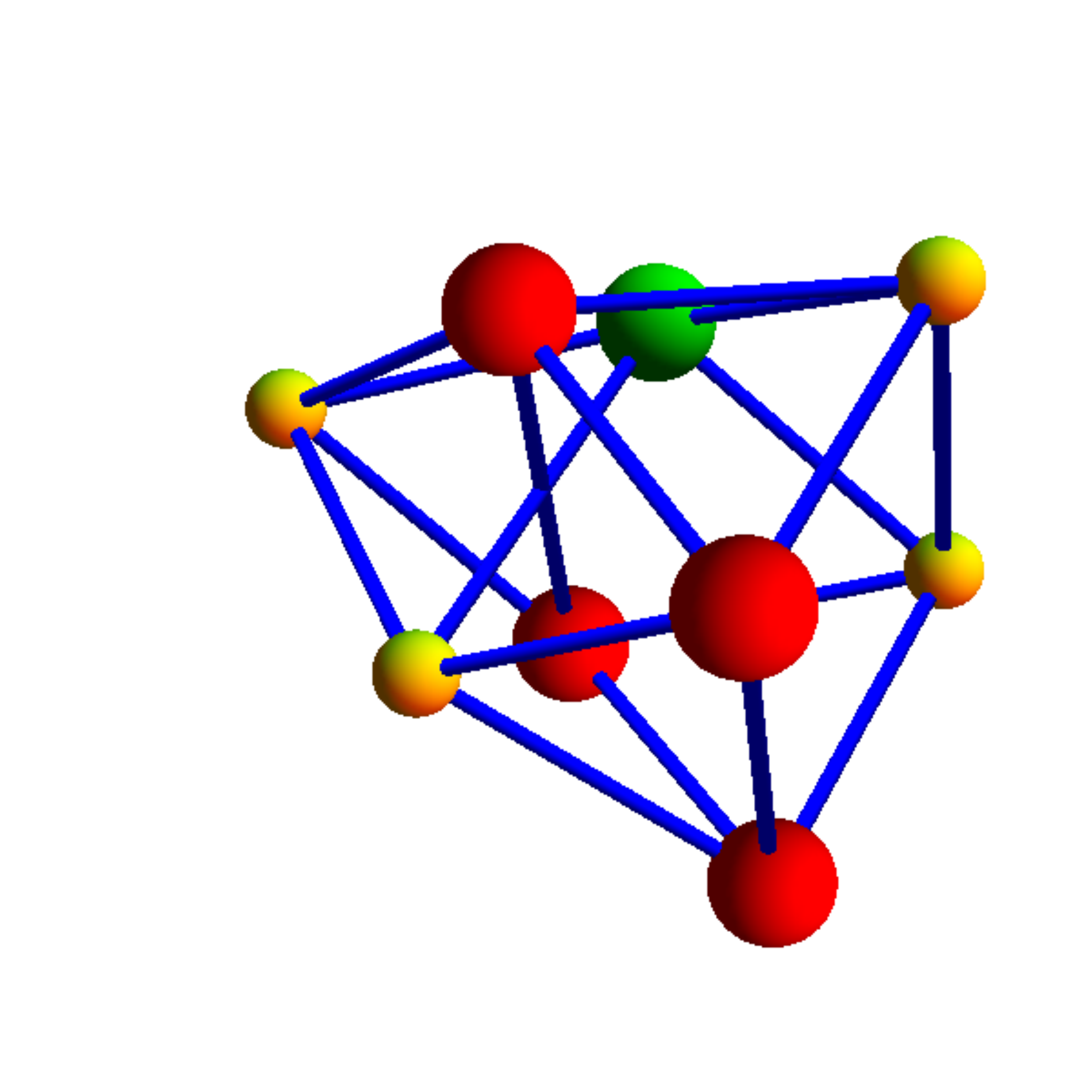}}
\caption{
The left picture shows a product graph $C_n \times C_m$. Such graphs have uniform dimension $1$
except for $n=m>3$ and curvature $-1$ leading to Euler characteristic $-n m$ due to the presence
of $n m$ holes in a torus. The right graph is a graph 
of constant dimension $2$ and where the curvature is constant $-1/3$ leading to Euler 
characteristic $-3$. These two examples illustrate that constant curvature and constant dimension
often imply each other. 
}
\label{complete}
\end{figure}

\section{Relations with the continuum}

The index definition can be used in the continuum too.
Let $M$ be a $d$-dimensional compact Riemannian manifold with Morse function $f$. Denote by
$S_r(p) = \{ q \; | \; d(p,q) = r \; \}$. For any point $p \in M$, define
$$  S_r^-(p) = \{ q \in S_r(p) \; | \; f(q)-f(p)<0 \; \} $$
and $i_{r,f}(p) = 1 - \chi(S_r^-)$.
For any Morse function $f$ on $M$ and critical point $p$, there exists $r_0$ such that for $0<r<r_0$
the index $i_{r,f}(p)$ agrees with the classical index of the gradient vector field $\nabla f$ at $p$.
The reason is that for a critical point of $f$ with
Morse index $k$ - the number of negative eigenvalues of the Hessian - 
then the sphere $S_r^-(p)$ is a $k$ dimensional sphere of
Euler characteristic $1+(-1)^k$. Therefore $i_{r,f}(p) = (-1)^k$.
This works also for non degenerate equilibrium points of vector fields $F$, where the index is either
$\pm 1$ \cite{Mil65}. The index works also for degenerate critical points of gradient fields,
at least in two dimensions. Vector fields in two dimensions with index larger than $1$ are not
gradient fields. \\

\begin{figure}
\scalebox{0.24}{\includegraphics{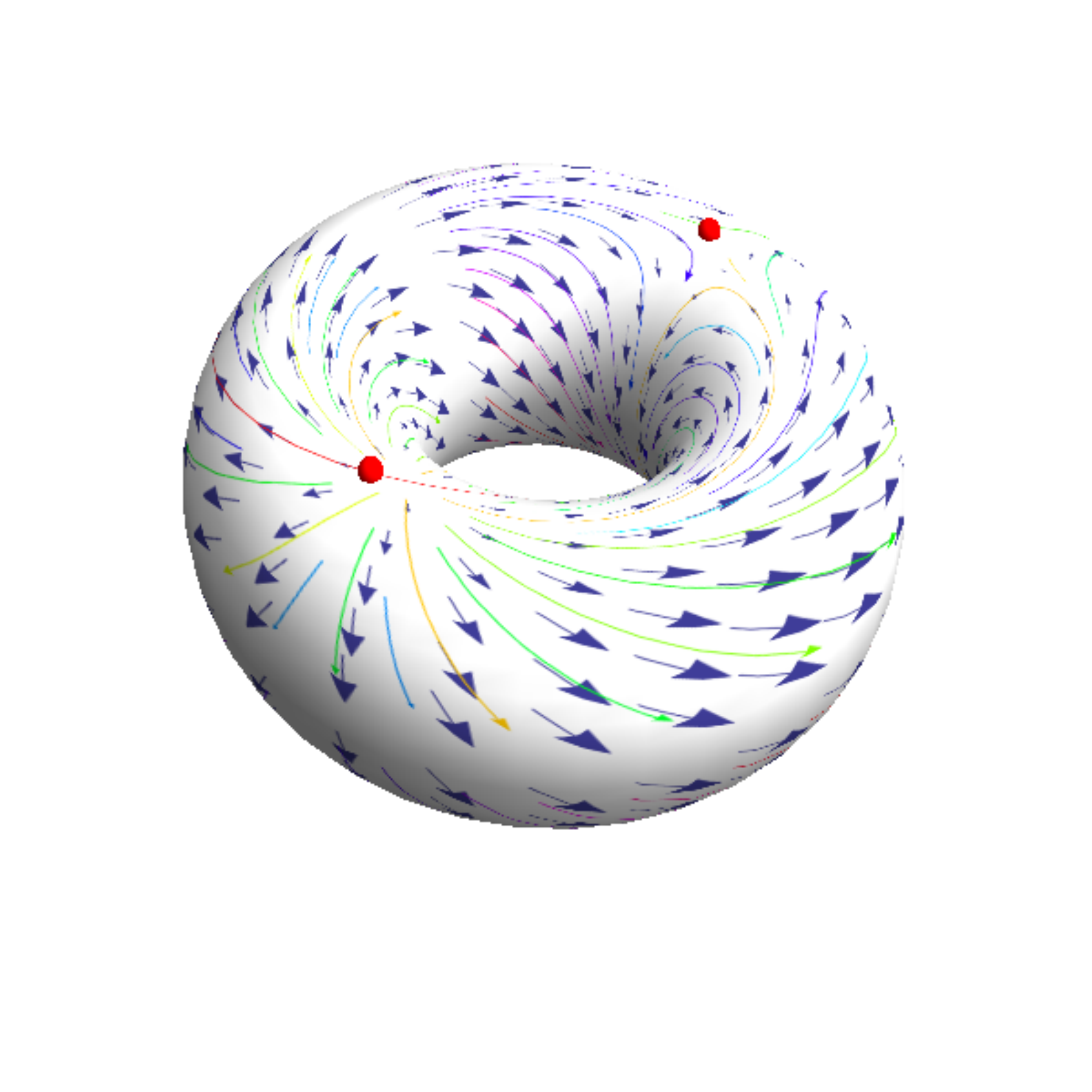}}
\caption{
The classical Poincar\'e Hopf theorem for a vector field on the torus.
There are 4 critical points, two with index $-1$ and two with index $1$.
The sum of the indices is $0$, which is the Euler characteristic of the
torus. 
}
\label{classical}
\end{figure}

We can compute the index of a classical Morse function $f$ on a manifold $M$ 
at a critical point $p$ by taking a sufficiently small $r>0$ and building 
a sufficiently fine triangularization of the sphere $S_r(p)$ and computing
the graph index at $p$ where the function $f$ is just the restriction of $f$ on $V$.
The reason is that by the Morse lemma, the function $f$ has in suitable coordinates the form
$f(x) = -\sum_{i=1}^k x_i^2 + \sum_{i=k+1}^n x_i^2$ near $p$. This means that $S^-(x)$ is homotopic
to a $(k-1)$-dimensional sphere of Euler characteristic $1-(-1)^k$. We have
$1-\chi(S^-(x)) = 1- (1-(-1)^k) = (-1)^k$. \\

We see therefore that the classical Poincar\'e Hopf theorem follows from the graph theoretical version
at least in the smooth case, where we can triangulate a manifold. Additionally, 
given a Morse function $f$ on $M$, we can triangularize in such a way that $f$ remains a Morse function 
on the graph, such that critical points of $f$ are vertices and such that these points are the only 
points with nonzero index in the graph. To see the later, note that near a point $q$ where $f$ has no critical point,
the gradient $\nabla f$ points in one direction. We just have to make sure that no two points are
on the same level curve of $f$. For each vertex $v$ there, the level curve through $v$ is a demarcation
line and  $S^-(v) = S(v) \cap \{x \; | \;  f(x)>f(v) \; \}$ has Euler characteristic $1$, 
as a triangularization of a half sphere always has in any dimension. 

\vspace{12pt}
\bibliographystyle{plain}

\end{document}